\documentclass[12pt]{article}
\usepackage{graphicx}
\usepackage{latexsym,amssymb}
\usepackage{amsthm}
\usepackage{indentfirst}
\usepackage{amsmath}
\usepackage{mathrsfs}
\usepackage{enumerate}
\usepackage{color}
\usepackage{xcolor}
\usepackage{hyperref}
\usepackage{fourier}
\usepackage{verbatim}
\usepackage[all]{xy}

\usepackage[a4paper,text={160true mm,230true mm},centering,top=35true mm,head=5true mm,headsep=2.5true mm,foot=8.5true mm]{geometry}

\newtheorem{theorem}{Theorem}[section]
\newtheorem{definition}{Definition}[section]
\newtheorem{proposition}[theorem]{Proposition}
\newtheorem{lemma}[theorem]{Lemma}
\newtheorem{corollary}[theorem]{Corollary}

\newtheorem{claim}[theorem]{Claim}

\newtheorem{theoremalph}{Theorem}
\newtheorem{corollaryalph}{Corollary}

 \def\NN{{\mathbb N}}

 \def\ZZ{{\mathbb Z}}

   \def\cM{{\cal M}} 
    
\def\cC{{\cal C}}   \def\cO{{\cal O}} \def\U{{\cal U}}
   \def\cP{{\cal P}} 
    
   \def\cR{{\cal R}}

\def\dim{\operatorname{dim}}

\def\supp{\operatorname{Supp}}

\begin{document}
\begin{sloppypar}
\title{Stable manifolds, Horseshoes and Lyapunov exponents for $C^1$ diffeomorphisms without domination}

\author{Yongluo Cao, Zeya Mi and Rui Zou\footnote{R. Zou is the corresponding author. Y. Cao is partially supported by National Key R\&D Program of China (2022YFA1005802) and NSFC 12371194. Z. Mi is partially supported by National Key R\&D Program of China (2022YFA1007800) and NSFC 12271260. R. Zou is partially supported by National Key R\&D Program of China (2022YFA1007800) and NSFC (nos 12471185, 12271386).
}}

\date{}
\maketitle

\begin{abstract} 
We develop the nonuniformly hyperbolic theory for $C^1$ diffeomorphisms admitting continuous invariant splitting without domination. 
This framework includes stable manifold theorems, shadowing   and closing lemmas, the existence of horseshoes and the approximation of Lyapunov exponents. 
The foundation is a new family of resonance blocks, each arising as the forward limit set of a typical point at carefully chosen resonance times where expansion, contraction and a weak scale-dependent domination coexist.
\end{abstract}
\tableofcontents

\section{Introduction}\label{SEC:1}

Nonuniformly hyperbolic theory, or Pesin theory, developed by Pesin, Katok, and others (see e.g. \cite{FY83,Ka80,P76,P77,R79,ly,ly2,BP07,LianLS12}) provides a firm foundation for understanding chaotic phenomena in smooth dynamical systems.  This theory classically requires $C^{1+\alpha}$ regularity ($\alpha>0$).
A long-standing challenge is to decide whether the theory survives when the \(C^{1+\alpha}\) hypothesis is relaxed to mere \(C^{1}\).  The present paper answers this question affirmatively for the  natural class of \(C^{1}\) systems ever considered: those possessing a continuous invariant splitting.  
 
 That \(C^{1}\) regularity alone is insufficient was shown  by Pugh \cite{P84} and  Bonatti-Crovisier-Shinohara \cite{BC13}: non-zero Lyapunov exponents do not guarantee the existence of stable manifolds.
 In response,  a parallel ``\(C^{1}\)+domination" theory has emerged in which the lacking \(C^{1+\alpha}\) smoothness is replaced by a uniform domination assumption.  
Within that framework,  Ma\~n\'e  \cite{M84} announced and      Abdener-Bonatti-Crovisier\cite{ABC} and Avila-Bochi\cite{AB}  proved  the  stable manifold theorem; 
 entropy formulas were   studied by Tahzibi\cite{T02}, Sun-Tian \cite{ST12}, Wang-Wang-Zhu \cite{WWZ} and Gan-Tong-Yang \cite{GYY25} recently;  
 horseshoe approximations  were established by by Gelfert\cite{Ge16} and Wang-Cao-Zou \cite{WCZ21}. 
Sun-Tian \cite{ST14} further enlarged the scope to the Liao-Pesin set, a unified structure that contains both dominated and limit-dominated splittings, and  developed the Pesin theory inside this class.  A natural question arises: 
\begin{center}
	\emph{Can the   nonuniform hyperbolicity theory be recovered for   $C^1$
	systems  {\bf without} any domination hypothesis?}
\end{center}

Our starting point is the observation made in \cite{CMZ25}: even in the complete absence of domination, (un)stable manifolds exist for $C^1$ diffeomorphisms as soon as the Oseledets splitting is continuous.
This   fact re-opens the door to the entire Pesin programme.
Here we transform that mere existence statement into a full Pesin-type theory.  Under the assumption that a  $C^1$ diffeomorphism  admits a continuous invariant splitting without domination,
we prove:  

\begin{itemize}
	\item[\bf (I)] A \(C^{1}\) stable manifold theorem (Theorem~\ref{main:stable}) that  recovers all essential features of the classical \(C^{1+\alpha}\) result, thereby strengthening the prior existence-only result \cite[Theorem 1.1]{CMZ25}. 
	
	\item[\bf (II)] Shadowing and closing lemmas (Theorem~\ref{main:shadowing}):  every sufficiently long pseudo-orbit in the regular blocks is shadowed by a unique genuine orbit, and periodic pseudo-orbits are closed by hyperbolic periodic points.
	
	\item[\bf (III)]  Existence of horseshoes(Theorem~\ref{main:horseshoe}), yielding the exponential growth-rate  
	   of periodic points (Corollary~\ref{Cor periodic}).
	
	 	\item[\bf (IV)] Approximation of Lyapunov exponents by those of periodic points and horseshoes (Theorems~\ref{periodic approximation} and~\ref{horseshoe approximation}).
\end{itemize}

Taken together, these results extend the framework of nonuniform hyperbolicity to the ``$C^1+$continuous" setting in a nearly optimal form.
The key mechanism behind the extension is a novel family of \emph{resonance blocks} $\{\Lambda_t\}_{t\ge 1}$ (see Section \ref{sec:block}), which builds upon and significantly refines the approach introduced in \cite{CMZ25}. Specifically, each
\(\Lambda_{t}\) is built from the gaps between Lyapunov exponents and captures, at a precise scale \(t\), coherent expansion and contraction 
along the forward orbit of typical points. The blocks also provide average domination over the recurrence period, guaranteeing that every admissible manifold expands/contracts uniformly within that window.  
This combination of properties allows us to fully reproduce the classical Pesin block philosophy within the $C^1$ setting.

\section{Statement of Main results}
{We state now the main results of our paper. Throughout, let $M$ be a compact Riemmannian manifold without boundary. Let $d$ denote the distance on $M$ induced by its Riemannian metric. Denote by ${\rm Diff}^1(M)$ the space of $C^1$ diffeomorphisms endowed with the usual $C^1$-topology. For $f\in {\rm Diff}^1(M)$, let $\cM_f(M)$ denote the space of $f$-invariant Borel probability measures on $M$.
Let $\{\varphi_j\}_{j=1}^{\infty}$ be  a countable dense subset of the unit sphere  in $C(M)$, which gives the distance $D$ of $\cM_f(M)$ given by
$$
D(\mu,\nu)=\sum_{j=1}^{\infty}\frac{|\int \varphi_jd\mu-\int \varphi_jd\nu|}{2^j}, \quad \forall \mu,\nu\in \cM_f(M).
$$


}

\subsection{Stable and unstable manifolds}

Let \(f \in \mathrm{Diff}^1(M)\), \(\mu \in \mathcal{M}_f(M)\), and let \(E \subset TM\) be a \(Df\)-invariant subbundle.  
By the Subadditive Ergodic Theorem, for \(\mu\)-almost every \(x \in M\) the minimal and maximal Lyapunov exponents along \(E\) are defined respectively by

\[
\chi_E^-(x,f)=\lim_{n\to +\infty}\frac{1}{n}\log m\bigl(Df^n|_{E(x)}\bigr), \qquad
\chi_E^+(x,f)=\lim_{n\to +\infty}\frac{1}{n}\log\bigl\|Df^n|_{E(x)}\bigr\|,
\]
where \(m(L)=\inf_{\|v\|=1}\|Lv\|=\|L^{-1}\|^{-1}\) denotes the co-norm of a linear isomorphism \(L\).  
When \(\mu\) is ergodic, the functions \(\chi_E^-(\,\cdot\,,f)\) and \(\chi_E^+(\,\cdot\,,f)\) are \(\mu\)-a.e. constant; their essential values are denoted by \(\chi_E^-(\mu,f)\) and \(\chi_E^+(\mu,f)\).

For a subset $\U\subset M$, we say that the splitting 
$
T_{\U}M=E^1\oplus E^2
$
is a \emph{continuous invariant splitting} if both $E^1$ and $E^2$ are continuous subbundles over $\U$, and for i=1,2, 
$$
Df_x(E^i_x)=E^i_{f(x)}~ \textrm{if}~x, f(x)\in \U.
$$

\begin{theoremalph}\label{main:stable}
Let \( f \in \mathrm{Diff}^1(M) \) preserving an ergodic measure $\mu$. Suppose there exists an open neighborhood \( \mathcal{U} \) of \( \mathrm{supp}(\mu) \) admitting a continuous \( Df \)-invariant splitting:
\(
T_{\mathcal{U}} M = E \oplus F.
\)

If the Lyapunov exponents satisfy
\[
\chi_F^{+}(\mu, f) < \min\left\{ 0,\, \chi_E^{-}(\mu, f) \right\},
\]
then for every sufficiently small \( \varepsilon > 0 \), there exist Borel functions \( \delta(x) > 0 \) and \( A(x) > 0 \) such that for \( \mu \)-almost every \( x \), there is a local stable manifold \( W^F_{\mathrm{loc}}(x, f) \) satisfying the following:
\begin{enumerate}[(1)]
    \item\label{sidf} \( W^F_{\mathrm{loc}}(x, f) \) is an embedded \( C^1 \)-disk centered at \( x \), with radius at least \( \delta(x) \), and tangent to \( F \).
    \item\label{contsf} For all \( y, z \in W^F_{\mathrm{loc}}(x, f) \) and all \( n \geq 1 \),
    \[
    d(f^n(y), f^n(z)) \leq A(x) \cdot {\rm e}^{(\chi_F^{+}(\mu, f) + \varepsilon)n} \cdot d(y, z).
    \]
    \item\label{subexp} The sequences \( \{\delta(f^n(x))\}_{n \in \mathbb{Z}} \) and \( \{A(f^n(x))\}_{n \in \mathbb{Z}} \) vary sub-exponentially along the orbit for \( \mu \)-almost every \( x \).
\end{enumerate}
Moreover, for \( \mu \)-almost every \( x \), the global stable manifold
\[
W^F(x, f) = \bigcup_{n \geq 0} f^{-n}\left( W^F_{\mathrm{loc}}(f^n(x), f) \right)
\]
is an injectively immersed \( C^1 \) submanifold tangent to \( F \), and admits the following characterization:
\begin{equation}\label{charc}
W^F(x, f) = \left\{ y \in M : \limsup_{n \to +\infty} \frac{1}{n} \log d(f^n(x), f^n(y)) \leq \chi_F^{+}(\mu, f) \right\}.
\end{equation}
\end{theoremalph}

This result refines our previous work \cite{CMZ25} in two ways: first, it recovers the sub-exponential decay of the rate function \(A(x)\) and the size \(\delta(x)\) along the orbit, thereby answering affirmatively the question raised in \cite[Page 8]{CMZ25}; second, it yields the global characterization \eqref{charc} of the stable manifold.

Likewise, under the assumption 
$
\chi_E^-(\mu,f)>\max\{0,\chi_F^+(\mu,f)\},
$
we can also establish the similar result on unstable manifolds tangent to $E$. 
Furthermore, property \eqref{charc} enables a simpler derivation of multi-level (un)stable manifolds, streamlining the approach taken in \cite[Theorem 3.3]{CMZ25}. 

\subsection{Shadowing lemmas  and horseshoe's approximation}

Given a subset $Q\subset M$ and a constant $\beta>0$, a sequence $\{(x_k,n_k)\}_{k\in\ZZ}$ is called a {\em $\beta$-pseudo-orbit for $f$ in $Q$} if for all $k \in \mathbb{Z}$,
$$
x_k, f^{n_k}(x_k)\in Q, \quad \textrm{and}~~ d(f^{n_k}x_k, x_{k+1})<\beta.
$$
For constants $\lambda>0$ and $\eta>0$, an orbit $\cO(z)$ is said to {\em $(\eta,\lambda)$-shadow }the pseudo-orbit $\{(x_k,n_k)\}_{k\in \ZZ}$ if   for all $k \in \mathbb{Z}$ and all $0 \le j \le n_k$,
$$
d(f^{s_k+j}(z),f^j(x_k))\le \eta\cdot {\rm e}^{-\lambda\cdot \min\{j,n_k-j\}},
$$
where 
$$
\displaystyle
s_k=
\left\{
\begin{array}{ll}
  \sum_{i=0}^{k-1}n_i, &\quad \textrm{if}~~k>0, \\[2ex]
 \quad 0, &\quad \textrm{if}~~k=0,\\[2ex]
  -\sum_{i=k}^{-1}n_i, & \quad \textrm{if}~~k<0.
\end{array}
\right.
$$

\begin{theoremalph}\label{main:shadowing}
Let $f\in {\rm Diff}^1(M)$ preserving an ergodic measure $\mu$. Assume that $\U$ is an open neighborhood of ${\rm supp}(\mu)$, for which there exists a continuous invariant splitting $T_{\U}M=E\oplus F$ satisfying $\chi_E^{-}(\mu,f)>0>\chi_F^{+}(\mu,f)$. Then there exists an increasing sequence of subsets $\{{Q}_t\}_{t\ge 1}$ such that $\mu(\cup_{t\ge 1}Q_t)=1$ and the following property holds:

For every $t\ge 1$, there exist $\widehat{\beta}=\widehat{\beta}(t)>0$, $\widehat{C}=\widehat{
C}(t)>0$, and $\widehat{N}=\widehat{N}(t)\in \NN$ such that for every $0<\beta<\widehat{\beta}$, if $\{(x_k,n_k)\}_{k\in \ZZ}$ is a $\beta$-pseudo-orbit in $Q_t$ with $n_k\ge \widehat{N}$ for all $k$, then there exists a unique point $z\in M$ whose orbit $(\widehat{C}\beta,\lambda)$-shadows $\{(x_k,n_k)\}_{k\in\ZZ}$ for some  $\lambda>0$. 

Moreover, if the pseudo-orbit $\{(x_k,n_k)\}_{k\in \ZZ}$ is periodic, then the shadowing point $z$ can be chosen to be a hyperbolic periodic point.
\end{theoremalph}


{

Recall that a compact invariant set \(K\) is called a {\em  horseshoe} for a diffeomorphism \(f\) if there exist integers \(m \geq 1\) and \(\ell \geq 2\), and pairwise disjoint closed sets \(K_0, \dots, K_{m-1}\), such that
\[
K = \bigcup_{i=0}^{m-1} K_i,\quad f^m(K_i)=K_i,\quad f(K_i)=K_{i+1 \bmod m},
\]
and \(f^m|_{K_0}\) is topologically conjugate to the full two-sided shift on \(\ell\) symbols.
 Let $h_{\mathrm{top}}(f|_{K})$ denote the topological entropy of $f$ restricted to $K$. For any $\mu\in \cM_f(M)$, we write $h_{\mu}(f)$ for its metric entropy. 

The following result extend the celebrated work \cite{Ka80,KH}.}

\begin{theoremalph}\label{main:horseshoe}
Let \( f \in \mathrm{Diff}^1(M) \)  preserving an ergodic measure $\mu$. Suppose there exists an open neighborhood \( \mathcal{U} \) of \( \mathrm{supp}(\mu) \) admitting a continuous \( Df \)-invariant splitting:
$
T_{\mathcal{U}} M = E \oplus F.
$
If 
$$
h_\mu(f) > 0,\quad \text{and}\quad
\chi_E^{-}(\mu,f) > 0 > \chi_F^{+}(\mu,f),
$$
then for every \( \varepsilon > 0 \), there exists a hyperbolic horseshoe \( K_\varepsilon \) with the following properties:
\begin{enumerate}[(1)]
    \item \label{ho1} \( |h_{\mathrm{top}}(f|_{K_{\varepsilon}}) - h_\mu(f)| < \varepsilon \),
    \item \label{ho2} \( d_H(K_\varepsilon, \mathrm{supp}(\mu)) < \varepsilon \),
    \item \label{ho3} For every \( f \)-invariant measure \( \nu \) supported on \( K_\varepsilon \), we have \( D(\mu, \nu) < \varepsilon \),
    \item \label{ho4} There exists \( m \in \mathbb{N}_+ \) such that for every \( x \in K_\varepsilon \),
    \begin{align*}
        \|Df^m_x(u)\| &\geq {\rm e}^{(\chi_E^-(\mu,f) - \varepsilon)m} \|u\|, \quad \forall u \in E(x), \\
        \|Df^m_x(v)\| &\leq {\rm e}^{(\chi_F^+(\mu,f) + \varepsilon)m} \|v\|, \quad \forall v \in F(x).
    \end{align*}
\end{enumerate}
\end{theoremalph}

As a consequence of Theorem \ref{main:horseshoe}, we obtain the next result
\begin{corollaryalph}\label{Cor periodic}
Let $f\in {\rm Diff}^1(M)$  preserving an ergodic measure $\mu$. Assume that $\U$ is an open neighborhood of ~${\rm supp}(\mu)$, for which there exists a continuous invariant splitting $T_{\U}M=E\oplus F$. If $h_{\mu}(f)>0$ and $\chi_E^{-}(\mu,f)>0>\chi_F^{+}(\mu,f)$, then
$$
\limsup_{n\to \infty}\frac{1}{n}\log \#P_n(f)\ge h_{\mu}(f),
$$
where $P_n(f)$ denotes the set of $n$-periodic points. 
\end{corollaryalph}

\subsection{Extended continuity of Oseledets splitting}

Let \(f\in\mathrm{Diff}^1(M)\) preserving an ergodic measure \(\mu\).  
The Oseledets theorem \cite{O68} provides an \(f\)-invariant Borel set \(\Gamma\subset M\) of full \(\mu\)-measure, together with a measurable \(Df\)-invariant splitting
$
T_{\Gamma}M=E_1\oplus\cdots\oplus E_{\ell},
$
such that for every \(x\in\Gamma\), every \(1\le i\le \ell\) and every non-zero vector \(v\in E_i(x)\),
\[
\lim_{n\to\pm\infty}\frac1n\log\|Df^n_x(v)\|=\chi_i(\mu).
\]
The constants \(\chi_1(\mu)>\cdots>\chi_{\ell}(\mu)\) are the {\em Lyapunov exponents\/} of \(\mu\), the decomposition \(T_{\Gamma}M=E_1\oplus\cdots\oplus E_{\ell}\) is the {\em Oseledets splitting}, and \(\Gamma\) is called the {\em regular set\/} of \(\mu\). A measure \(\mu\) is called {\em hyperbolic} if all of its Lyapunov exponents are non-zero.

In the next two theorems we assume that the Oseledets splitting \(T_{\Gamma}M=E_1\oplus\cdots\oplus E_{\ell}\) extends to a continuous \(Df\)-invariant splitting on a neighbourhood of \(\overline{\Gamma}\).   Under this assumption we establish approximation of the Lyapunov exponents by those of periodic orbits and by horseshoes.
\begin{theoremalph}\label{periodic approximation}
Let $f\in {\rm Diff}^1(M)$  preserving an ergodic hyperbolic measure $\mu$. Assume that the Oseledets splitting $T_{\Gamma}M=E_1\oplus \cdots \oplus E_{\ell}$  admits a continuous \(Df\)-invariant extension over a neighborhood $\mathcal{U}$ of  ~$\overline{\Gamma}$. Then, for any $\varepsilon>0$, there exists a periodic point $p\in M$ such that 
$$
|\lambda_i(\mu)-\lambda_i(\mu_p)|<\varepsilon,\quad \forall 1\le i \le d,
$$
where $\lambda_1(\mu)\ge \cdots \ge \lambda_d(\mu)$ are Lyapunov exponents of $
\mu$ counted with multiplicities and $\lambda_i(\mu_p)$ those of   the orbit measure of $p$.
\end{theoremalph}

\begin{theoremalph}\label{horseshoe approximation}
	Let $f\in {\rm Diff}^1(M)$  preserving an ergodic hyperbolic measure $\mu$. Assume that the Oseledets splitting $T_{\Gamma}M=E_1\oplus \cdots \oplus E_{\ell}$  admits a continuous invariant extension over a neighborhood $\mathcal{U}$ of ~$\overline{\Gamma}$. If $h_{\mu}(f)>0$, then for every $\varepsilon>0$, there exists a hyperbolic horseshoe $K_{\varepsilon}$ with the following properties:
	\begin{enumerate}[(1)]
		\item $|h_{top}(f|_{K_{\varepsilon}})-h_{\mu}(f)|<\varepsilon$.
		\item $d_H(K_{\varepsilon}, {\rm supp}(\mu))<\varepsilon$.
		\item for every $f$-invariant measure $\nu$ supported on $K_{\varepsilon}$, we have $D(\mu,\nu)<\varepsilon$.
		\item there exists $m\in \NN$ such that for every $x\in K_{\varepsilon}$ and   $1\le j\le \ell,$
		$$
	{\rm e}^{(\chi_j(\mu)-\varepsilon)m}\|u\|\le 	\|Df^m_x(u)\|\le {\rm e}^{(\chi_j(\mu)+\varepsilon)m}\|u\|,\quad \forall u\in E_j(x).
		$$
	\end{enumerate}
\end{theoremalph}

\section{Resonance blocks}\label{sec:block}

In this section, we introduce a family of new blocks for diffomorphisms admitting two continuous invariant subbundles on the support of some ergodic measure. This new construction plays a crucial role in proving our main results.  

Let $g\in {\rm Diff}^1(M)$, and $\nu$ be a $g$-ergodic measure with a continuous invariant splitting 
$
T_{{\rm supp}(\nu)}M=E\oplus F.
$
Given $\varepsilon>0$, $x\in {\rm supp}(\nu)$ and $n\in \NN$, define
\begin{align*}
a^1_n(\nu,\varepsilon,x,g)=&\sup_{0\le j \le n}\left\{\prod_{i=n-j}^{n-1}\frac{\|Dg|_{F(g^{i}(x))}\|}{{\rm e}^{\chi_F^+(\nu,g)+\varepsilon}},\quad \prod_{i=n-j}^{n-1}\frac{{\rm e}^{\chi_E^-(\nu,g)-\varepsilon}}{m(Dg|_{E(g^{i}(x))})}\right\}, \\
a^2_n(\nu,\varepsilon,x,g)=&\sup_{j\ge 0}\left\{\prod_{i=n}^{n+j-1}\frac{\|Dg|_{F(g^{i}(x))}\|}{{\rm e}^{\chi_F^+(\nu,g)+\varepsilon}}, \quad \prod_{i=n}^{n+j-1}\frac{{\rm e}^{\chi_E^-(\nu,g)-\varepsilon}}{m(Dg|_{E(g^{i}(x))})}\right\},
\end{align*}
where we set $\prod_{i=n}^{n-1}a_i=1$ for any sequence $\{a_i\}$ of real numbers.
We see from definition that $a^i_n(\nu,\varepsilon,x,g)\ge 1$ for each $1\le i \le 2$.
Define
$$
H_t(\nu,\varepsilon,x,g)=\left\{n\in \NN: a^i_n(\nu,\varepsilon,x,g)\le t, \quad \forall 1\le i \le 2 \right\},\quad t\ge 1.
$$
Let us consider the forward limit set of $x$ with times restricted in $H_t(\nu,\varepsilon,x,g)$, defined as follows
$$
\Lambda_t^{\varepsilon}(\nu,x,g)=\bigcap_{m\ge 1}\overline{\bigcup_{j\ge m}\left\{g^j(x):j\in H_t(\nu,\varepsilon,x,g)\right\}},\quad t\ge 1.
$$

\begin{definition}
We call $\Lambda_{t}^{\varepsilon}(\nu,x,g)$ a {\em resonance block} with respect to $(g,\nu, E\oplus F)$.
\end{definition}

\subsection{Basic properties on resonance blocks}
This subsection gives several basic properties on the resonance blocks that will be used throughout this paper.  Let us fix a $g$-ergodic measure $\nu$ exhibiting continuous splitting $T_{{\rm supp}(\nu)}M=E\oplus F$. Given $\varepsilon>0$, $x_0\in {\rm supp}(\nu)$, we set for simplicity that
\begin{itemize}
\item $\chi_E^-:=\chi_E^-(\nu,g)$,\quad $\chi_F^+:=\chi_F^+(\nu,g)$.
\item $H_t:=H_t(\nu,\varepsilon,x_0,g),\quad \Lambda_t:=\Lambda_t^{\varepsilon}(\nu,x_0,g),\quad \forall t\ge 1$.
\item $a_{n}^i:=a_{n}^i(\nu,\varepsilon,x_0,g),\quad \forall n\in \NN, \quad i=1, 2.$
\end{itemize}

Given a subset $\mathbb{J}\subset \NN$, we define its lower density and upper density as follows:
$$
\underline{{\rm Dens}}(\mathbb{J}):=\liminf_{n\to +\infty}\frac{1}{n}\#(\mathbb{J}\cap \{0,\dots,n-1\}),
$$
$$
\overline{{\rm Dens}}(\mathbb{J}):=\liminf_{n\to +\infty}\frac{1}{n}\#(\mathbb{J}\cap \{0,\dots,n-1\}).
$$
We just write ${\rm Dens}(\mathbb{J})$ when the limit exists.

\begin{lemma}\label{basicblock}
We have the following properties on $\Lambda_t$ and $H_t$, $t\ge 1$: 
\begin{enumerate}[(1)]
\item\label{lambda:incr} $\{\Lambda_t\}_{t\ge 1}$ is an increasing sequence of compact subsets.
\item\label{lambda:mden} If $x_0$ is $\nu$-generic, then 
$
\nu(\Lambda_t)\ge \overline{{\rm Dens}}(H_t)
$
for every $t\ge 1$.
\item\label{lambda:inv} There exists $c_0\ge 1$ such that $g\left(\Lambda_t\right)\subset \Lambda_{c_0t}$ for every $t\ge 1$.
\end{enumerate}
\end{lemma}
\begin{proof}
Item \eqref{lambda:incr} follows immediately from the definition. To show \eqref{lambda:mden}, let $x_0$ be $\nu$-generic, thus
$$
\nu_n=\frac{1}{n}\sum_{i=0}^{n-1}\delta_{g^ix_0}\to \nu \quad \textrm{as}~n\to +\infty.
$$
Then we conclude from definition that   
\begin{eqnarray*}
\nu(\Lambda_t)&=&\lim_{m\to \infty}\nu\left(\overline{\bigcup_{j\ge m}\left\{g^j(x_0):j\in H_t\right\}}\right)\\
&\ge & \lim_{m\to \infty}\limsup_{n\to \infty}\nu_n\left(\overline{\bigcup_{j\ge m}\left\{g^j(x_0):j\in H_t\right\}}\right)\\
&\ge& \lim_{m\to \infty}\limsup_{n\to \infty}\nu_n\left(\bigcup_{j\ge m}\left\{g^j(x_0):j\in H_t\right\}\right)\\
&=& \lim_{m\to \infty}\limsup_{n\to \infty}\frac{1}{n}\#(H_t\cap [0,n-1])\\
&=&\overline{{\rm Dens}}(H_t).
\end{eqnarray*}

We now prove \eqref{lambda:inv}.
It follows from the definition that
\begin{align}
a_{n+1}^1&=\max\left\{1, \max\left\{\frac{\|Dg|_{F(g^n(x_0))}\|}{{\rm e}^{\chi_F^++\varepsilon}}, \frac{{\rm e}^{\chi_E^--\varepsilon}}{m(Dg|_{E(g^n(x_0))})}\right\}\cdot a_n^1\right\},\label{eq:an1}\\
a_{n}^2&=\max\left\{1, \max\left\{\frac{\|Dg|_{F(g^n(x_0))}\|}{{\rm e}^{\chi_F^++\varepsilon}}, \frac{{\rm e}^{\chi_E^--\varepsilon}}{m(Dg|_{E(g^n(x_0))})}\right\}\cdot a_{n+1}^2\right\}.\label{eq:an2}
\end{align}
Define
$$
c_0=\max_{x\in {\rm supp}(\nu)}\left\{\frac{\|Dg|_{F(x)}\|}{{\rm e}^{\chi_F^++\varepsilon}},\frac{{\rm e}^{\chi_F^++\varepsilon}}{\|Dg|_{F(x)}\|},\frac{m(Dg|_{E(x)})}{{\rm e}^{\chi_E^--\varepsilon}},\frac{{\rm e}^{\chi_E^--\varepsilon}}{m(Dg|_{E(x)})},1\right\}.
$$
Then from \eqref{eq:an1} and \eqref{eq:an2}, and since $a^i_n\ge 1$, we have 
\begin{equation}\label{eq:an+1}
a_{n+1}^i\le c_0a_n^i\quad \forall n\in \NN, \quad i=1,2.
\end{equation}
Now fix $t\ge 1$ and let $x\in \Lambda_t$. By definition, there exists $\{n_k\}\subset H_t$ such that $g^{n_k}(x_0)\to x$ as $k\to +\infty$. Since $a_{n_k}^i \le t$ for $i = 1, 2$ and all $k \in \NN$ by the definition of $H_t$, inequality \eqref{eq:an+1} implies $a_{n_k+1}^i \le c_0 \cdot t$. As $g^{n_k+1}(x_0)\to g(x)$ as $k\to +\infty$, we see by definition that
$g(x)\in \Lambda_{c_0\cdot t}$. Hence, we get
$
g(\Lambda_t)\subset \Lambda_{c_0\cdot t}
$
for every $t\ge 1$.
\end{proof}

As a result of Lemma \ref{basicblock}, we have 
\begin{lemma}\label{pro:blockoff}
Assume $x$ is a $\nu$-generic point such that 
$
\overline{{\rm Dens}}(H_{t_0})>0
$
for some $t_0\ge 1$.
Then 
$$
\lim_{t\to +\infty}\nu\left(\Lambda_t\right)=\nu\left(\bigcup_{t\ge 1}\Lambda_t\right)=1.
$$
\end{lemma}

\begin{proof}
 By Item \eqref{lambda:incr} of Lemma \ref{basicblock}, the union $\bigcup_{t\ge 1}\Lambda_t$ is measurable. Since
$
{\rm Dens}(H_{t_0})>0.
$
Item \ref{lambda:mden} of Lemma \ref{basicblock} implies 
$
\nu(\Lambda_{t_0})>0.
$
By the ergodicity of $\nu$, it follows that 
$$
\nu\left(\bigcup_{n\ge 1}g^n(\Lambda_{t_0})\right)=1.
$$
Item \eqref{lambda:inv} of Lemma \ref{basicblock} suggests that there exists $c_0\ge 1$ such that $g^n(\Lambda_{t_0})\subset \Lambda_{c_0^nt_0}$ for all $n\ge 1$. Hence,
$$
\bigcup_{n\ge 1}g^n(\Lambda_{t_0})\subset \bigcup_{t\ge 1}\Lambda_t,
$$
and therefore
$
\nu\left(\bigcup_{t\ge 1}\Lambda_t\right)=1.
$
By the monotonicity of $\Lambda_t$ provided in Item \eqref{lambda:incr} of Lemma \ref{basicblock}), we get
$$
\lim_{t\to +\infty}\nu\left(\Lambda_t\right)=\nu\left(\bigcup_{t\ge 1}\Lambda_t\right)=1.
$$
\end{proof}

Let ${\mathcal{R}}=\bigcup_{t\ge 1}\Lambda_t$, define 
\begin{equation}\label{def:Tx}
T(x)=\inf\left\{t\ge 1: x\in \Lambda_t\right\},\quad x\in \mathcal{R}.
\end{equation}

\begin{lemma}\label{T measurable}
The function $T$ is lower semi-continuous on $\mathcal{R}$ and hence measurable.
\end{lemma}

\begin{proof}
Define
$$
\widetilde{\Lambda}_t=\bigcap_{s>t}\Lambda_s, \quad t\ge 1.
$$
Then, we know by definition that $\Lambda_t \subset \widetilde{\Lambda}_t$ for all $t \ge 1$, and $\widetilde{\Lambda}_{t_1} \subset \widetilde{\Lambda}_{t_2}$ whenever $t_1 < t_2$. 
Let 
$$
\widetilde{T}(x) = \inf \{ t \ge 1 : x \in \widetilde{\Lambda}_t \}\quad \textrm{for}~ x \in \mathcal{R}.
$$ 
We show that $T(x) = \widetilde{T}(x)$.
On the one hand, by the definition of $T(x)$ and the monotonicity of $\{\Lambda_t\}$, we have $x\in \Lambda_t$ for every $t>T(x)$. Thus, 
\begin{equation}\label{xlambda}
x\in \bigcap_{t>T(x)}\Lambda_t=\widetilde{\Lambda}_{T(x)},
\end{equation}
so $\widetilde{T}(x)\le T(x)$. Conversely, we know by definition that $x\in \widetilde{\Lambda}_t$ for every $t>\widetilde{T}(x)$. It follows that 
\begin{equation}\label{eq:tilde}
x\in \bigcap_{t>\widetilde{T}(x)}\widetilde{\Lambda}_t=\bigcap_{t>\widetilde{T}(x)}\bigcap_{s>t}\Lambda_s=\bigcap_{s>\widetilde{T}(x)}\Lambda_s=\widetilde{\Lambda}_{\widetilde{T}(x)}.
\end{equation}
Therefore, $x \in \Lambda_s$ for every $s > \widetilde{T}(x)$, and thus $T(x)\le s$ for every $s>\widetilde{T}(x)$. Hence, we know $T(x)\le \widetilde{T}(x)$ from the arbitrariness of $s$.
Therefore, $T(x)=\widetilde{T}(x)$. Moreover, \eqref{eq:tilde} also implies that 
$$
T(x)=\widetilde{T}(x)=\min\{t\ge 1: x\in \widetilde{\Lambda}_t\},
$$
so,
$$
\{x\in \mathcal{R}: T(x)\le t_0\}=\widetilde{\Lambda}_{t_0}.
$$
Observe that $\widetilde{\Lambda}_{t_0}$ is a closed subset, $T$ is lower semi-continuous and hence measurable. 
\end{proof}

The following result asserts that $T$ defined in \eqref{def:Tx} is tempered along the orbit for almost every point, which is essential in getting the sub-exponentially property on $\delta(x)$ and $A(x)$ in Theorem \ref{main:stable}.

\begin{lemma}\label{lem:temp}
If $\eta(\mathcal{R})=1$ for a $g$-invariant measure $\eta$, then
$$
\lim_{n\to \pm\infty}\frac{1}{|n|}\log T(g^n(x))=0,\quad \eta-\textrm{a.e.}~x.
$$
\end{lemma}

\begin{proof}
Let $x\in \mathcal{R}$. By \eqref{xlambda}, we have $x\in \bigcap_{t>T(x)}\Lambda_t$, so in particular $x\in \Lambda_{T(x)+\delta}$ for every $\delta>0$. Hence, there exists a sequence $\{n_k\}_{k\ge 1}\subset H_{T(x)+\delta}$ such that $g^{n_k}(x_0)\to x$ as $k\to +\infty$. It follows that
\begin{equation}\label{ineq:ank}
a_{n_k}^i\le T(x)+\delta,\quad  i=1,2.
\end{equation}
From \eqref{eq:an1}-\eqref{eq:an2}, for each $1\le i \le 2$ we obtain
$$
a_{n_k+1}^i\le c_0 a_{n_k}^i\le c_0\cdot (T(x)+\delta),
$$
so the sequences $\{a_{n_k}^i\}_{k\ge 1}$ and $\{a_{n_k+1}^i\}_{k\ge 1}$ are bounded from above.
Up to considering some subsequence, we may assume that for $i=1,2$, 
$$
\lim_{k\to +\infty}a_{n_k}^i=a^i,\quad \lim_{k\to +\infty}a_{n_k+1}^i=\widetilde{a}^i.
$$
The first convergence together with \eqref{ineq:ank} implies
\begin{equation}\label{ineq:ak}
a^i\le T(x)+\delta, \quad  i=1, 2.
\end{equation}
The second convergence implies that for every large enough $k$, 
$$
a_{n_k+1}^i<\widetilde{a}^i+\delta,\quad i=1,2.
$$
So,  $n_k+1\in H_{\max_{1\le i \le 2}\{\widetilde{a}^i+\delta\}}$ and $g^{n_k+1}(x_0)\to g(x)$ as $k\to +\infty$. Therefore, $g(x)\in \Lambda_{\max_{1\le i \le 2}\widetilde{a}^i+\delta}$, and 
\begin{equation}\label{ineq:tgb}
T(g(x))\le \max_{1\le i \le 2}\widetilde{a}^i+\delta.
\end{equation}
Now let $k\to +\infty$ in \eqref{eq:an1}-\eqref{eq:an2} with $n=n_k$ to obtain
\begin{align*}
\widetilde{a}^1&=\max\left\{1, ~\max\left\{\frac{\|Dg|_{F(x)}\|}{{\rm e}^{\chi_F^++\varepsilon}}, \frac{{\rm e}^{\chi_E^--\varepsilon}}{m(Dg|_{E(x)})}\right\}\cdot a^1\right\}, \\
a^2&=\max\left\{1, ~\max\left\{\frac{\|Dg|_{F(x)}\|}{{\rm e}^{\chi_F^++\varepsilon}}, \frac{{\rm e}^{\chi_E^--\varepsilon}}{m(Dg|_{E(x)})}\right\}\cdot \widetilde{a}^2\right\}.
\end{align*}
Define
$$
b(x)=\max\left\{\frac{{\rm e}^{\chi_E^--\varepsilon}}{m(Dg|_{E(x)})},\quad \frac{\|Dg|_{F(x)}\|}{{\rm e}^{\chi_F^++\varepsilon}}\right\}.
$$
Then from \eqref{ineq:ak} we deduce
$$
\widetilde{a}^1\le \max\{1,b(x)\cdot (T(x)+\delta)\},\quad
\widetilde{a}^2\le b(x)^{-1}\cdot (T(x)+\delta).
$$
Substituting into \eqref{ineq:tgb} gives
$$
T(g(x))\le \max\left\{1,b(x)\cdot (T(x)+\delta),b(x)^{-1}\cdot (T(x)+\delta)\right\}+\delta.
$$
Since $\delta>0$ is arbitrary, it follows that 
$$
T(g(x))\le \max\left\{1, \max\{b(x),b(x)^{-1}\}\cdot T(x)\right\}.
$$
Therefore
$$
\frac{T(g(x))}{T(x)}\le \max \left\{\frac{1}{T(x)}, \max\{b(x),b(x)^{-1}\}\right\}\le \max\{1,b(x),b(x)^{-1}\},
$$
and hence
$$
\left(\log T(g(x))-\log T(x)\right)^+\le |\log b(x)|.
$$
Since $b(x)$ is continuous, the right-hand side belongs to $L^1(\eta)$. Moreover, by $g$-invariance of $\eta$, we also have 
$$
\left(\log T(g^{-1}(x))-\log T(x)\right)^-=-\left(\log T(x)-\log T(g^{-1}(x))\right)^+\in L^1(\eta).
$$
An application of \cite[Lemma 3.4.3]{Ar98} now yields
$$
\lim_{n\to \pm \infty}\frac{1}{|n|}\log T(g^n(x))=0,\quad \eta-\textrm{a.e.}~x,
$$
as required
\end{proof}

\begin{lemma}\label{pro:fulles}
For every $x\in \mathcal{R}$ and $k\ge 1$, the following estimates hold:
\begin{enumerate}[(1)]
\item\label{la:pp1} 
$
\prod_{i=-k}^{-1}m(Dg|_{E(g^i(x))})\ge T(x)^{-1}\cdot {\rm e}^{(\chi_E^--\varepsilon)k}.
$
\item\label{la:pp2}
$
\prod_{i=0}^{k-1}m(Dg|_{E(g^i(x))})\ge T(x)^{-1}\cdot {\rm e}^{(\chi_E^--\varepsilon)k}.
$
\item\label{la:p3}
$
\prod_{i=-k}^{-1}\|Dg|_{F(g^i(x))}\|\le T(x)\cdot {\rm e}^{(\chi_F^++\varepsilon)k}.
$
\item\label{la:p4}
$
\prod_{i=0}^{k-1}\|Dg|_{F(g^i(x))}\|\le T(x)\cdot {\rm e}^{(\chi_F^++\varepsilon)k}.
$
\end{enumerate}

\end{lemma}

\begin{proof}
We prove property \eqref{la:pp1} in detail, properties \eqref{la:pp2}-\eqref{la:p4} follow by analogous arguments.
Since $x\in \bigcap_{t>T(x)}\Lambda_t$, we have $x\in \Lambda_{T(x)+\delta}$ for every $\delta>0$. Hence, there exists a sequence $\{n_k\}_{k\ge 1}\subset H_{T(x)+\delta}$ such that $g^{n_k}(x_0)\to x$ as $k\to \infty$. For every $m\ge 1$, choose $n_k>m$. Then 
$$
\prod_{i=n_k-m}^{n_k-1}m\left(Dg|_{E(g^ix_0)}\right)\ge \left(a_{n_k}^1\right)^{-1}\cdot {\rm e}^{(\chi_E^--\varepsilon)m}\ge (T(x)+\delta)^{-1}\cdot {\rm e}^{(\chi_E^--\varepsilon)m}.
$$
Taking $k\to \infty$ and then $\delta\to 0$, we obtain
$$
\prod_{i=-m}^{-1}m\left(Dg|_{E(g^ix)}\right)\ge T(x)^{-1}\cdot {\rm e}^{(\chi_E^--\varepsilon)m}.
$$
which completes the proof of \eqref{la:pp1}.
\end{proof}

It follows from the definition that $T(x)\le t$ on $\Lambda_t$ for every $t\ge 1$. Thus, Lemma \ref{pro:fulles} implies the following result.

\begin{corollary}\label{cor:lambda-bl}
For every $x\in \Lambda_t$ and $k\ge 1$, we have
\begin{enumerate}[(1)]
\item 
$
\prod_{i=-k}^{-1}m(Dg|_{E(g^i(x))})\ge t^{-1}\cdot {\rm e}^{(\chi_E^--\varepsilon)k}.
$
\item 
$
\prod_{i=0}^{k-1}m(Dg|_{E(g^i(x))})\ge t^{-1}\cdot {\rm e}^{(\chi_E^--\varepsilon)k}.
$
\item 
$
\prod_{i=-k}^{-1}\|Dg|_{F(g^i(x))}\|\le t\cdot {\rm e}^{(\chi_F^++\varepsilon)k}.
$
\item 
$
\prod_{i=0}^{k-1}\|Dg|_{F(g^i(x))}\|\le t\cdot {\rm e}^{(\chi_F^++\varepsilon)k}.
$
\end{enumerate}

\end{corollary}

We remark that, when assuming the hyperbolicity of $\nu$, such as $\chi_E^->0>\chi_F^+$, then Corollary \ref{cor:lambda-bl} demonstrates that for any sufficiently small $\varepsilon>0$, each block behaves like the usual block with uniform hyperbolicity and domination in weak sense. 

\subsection{Large weight on resonance blocks}

The following result ensures that for large iterates, there exists an increasing sequence of resonance blocks in measure as the level 
$t$ increases.

\begin{proposition}\label{pro:ftoblock}
Let $f\in {\rm Diff}^1(M)$ and $\mu$ be an $f$-ergodic measure exhibiting continuous splitting $T_{{\rm supp}(\mu)}M=E\oplus F$. Then for any $\varepsilon>0$, there exist $N\in \NN$, an ergodic component $\nu$ of $\mu$ with respect to $f^N$ and a $\nu$-generic point $x_0$ for $f^N$ such that 
$$
\lim_{t\to +\infty}\nu\left(\Lambda_t^{N\varepsilon}(\nu,x_0,f^N)\right)=\nu\left(\bigcup_{t\ge 1}\Lambda_t^{N\varepsilon}(\nu,x_0,f^N)\right)=1.
$$
\end{proposition}

To show this proposition, we need the following result.
\begin{lemma}\label{lem:fpl}
Let $\gamma_1<\gamma_2$ and $\Phi=\{\varphi_n\}_{n\ge 1}$ be a sequence of sub-additive continuous functions. Suppose $\mu$ is an $f$-invariant measure such that 
$$
\lim_{n\to \infty}\frac{\varphi_n(x)}{n}\le \gamma_1,\quad \mu-\textrm{a.e.}~ x.
$$
Then for every $0<\theta<1$, there exists $N_0\ge 1$ such that for all $N\ge N_0$,  
$$
\mu\left(\{x: {\rm Dens}(T_N(\Phi,x,\gamma_2))>\theta\}\right)>\theta,
$$
where 
\begin{align*}
T_N(\Phi,x,\gamma_2)=\Big\{n\in \NN &: \frac{1}{kN}\sum_{i=n-k}^{n-1}\varphi_N(f^{iN}(x))\le \gamma_2,\\
&\quad \frac{1}{jN}\sum_{i=n}^{n+j-1}\varphi_N(f^{iN}(x))\le \gamma_2,\quad \forall 1\le k \le n,\quad \forall j\ge 1\Big\}.
\end{align*}
\end{lemma}

\begin{proof}
For each $N\in \NN$, define 
\begin{align*}
T^1_N(\Phi,x,\gamma_2)&=\left\{n\in \NN: \frac{1}{kN}\sum_{i=n-k}^{n-1}\varphi_N(f^{iN}(x))\le \gamma_2,\quad \forall 1\le k \le n\right\}, \\
T^2_N(\Phi,x,\gamma_2)&=\left\{n\in \NN: \frac{1}{jN}\sum_{i=n}^{n+j-1}\varphi_N(f^{iN}(x))\le \gamma_2,\quad \forall j\ge 1\right\}.
\end{align*}
Then $T_N(\Phi,x,\gamma_2)=T^1_N(\Phi,x,\gamma_2)\cap T^2_N(\Phi,x,\gamma_2)$. By \cite[Lemma 3.4]{CMZ25}, for every $\theta\in (0,1)$, there exists $N_0\in \NN$ such that 
\begin{equation}\label{theta0}
\mu\left(\left\{x: {\rm Dens}\footnotemark(T_N^1(\Phi,x,\gamma_2))>\frac{1+\theta}{2}\right\}\right)>\frac{1+\theta}{2},\quad \forall N\ge N_0.
\end{equation}
\footnotetext{We strengthen $\underline{{\rm Dens}}$ in \cite[Lemma 3.4]{CMZ25} to ${\rm Dens}$ since the density limit exists $\mu$-a.e. by the Birkhoff ergodic theorem.}

Now we estimate the density of $T_N^2(\Phi,x,\gamma_2)$. Consider the sets 
$$
H_N(\Phi,\gamma_2)=\left\{x:\frac{1}{jN}\sum_{i=0}^{j-1}\varphi_N(f^{iN}x)\le \gamma_2,\quad \forall j\ge 1\right\},\quad N \in \NN.
$$
By \cite[Lemma 3.6]{CMZ25}(or \cite[Lemma 2.7]{CZZ}), we know 
$$
\lim_{N\to +\infty}\mu(H_N(\Phi,\gamma_2))=1.
$$
Thus, by increasing $N_0$ if necessary, we may assume that 
$$
\mu(H_N(\Phi,\gamma_2))>1-\left(\frac{1-\theta}{2}\right)^2,\quad \forall N\ge N_0.
$$
For fixed $N\ge N_0$, define
$$
\psi_N(x)=\lim_{n\to +\infty}\frac{1}{n}\sum_{i=0}^{n-1}\chi_{H_N(\Phi,\gamma_2)}(f^{iN}x),
$$
which is well defined for $\mu$-a.e. $x$, guaranteed by the Birkhoff ergodic theorem. Moreover,
\begin{equation}\label{theta1}
\int \psi_N(x)d\mu(x)=\mu(H_N(\Phi,\gamma_2))>1-\left(\frac{1-\theta}{2}\right)^2.
\end{equation}
Let 
$$
\Gamma_N=\left\{x: \psi_N(x)>\frac{1+\theta}{2}\right\}.
$$
Since
$$
\int \psi_N(x)d\mu(x)\le \mu(\Gamma_N)+\frac{1+\theta}{2}\mu(M\setminus \Gamma_N),
$$
Inequality \eqref{theta1} implies 
$$
\mu(\Gamma_N)>\frac{1+\theta}{2}.
$$
Observe that
$$
{\rm Dens}(T_N^2(\Phi,x,\gamma_2))=\psi_N(x)>\frac{1+\theta}{2},\quad \forall x\in \Gamma_N.
$$
Thus, we conclude that
\begin{equation}\label{theta3}
\mu\left({\rm Dens}(T_N^2(\Phi,x,\gamma_2))>\frac{1+\theta}{2}\right)\ge \mu(\Gamma_N)>\frac{1+\theta}{2},\quad \forall N\ge N_0.
\end{equation}
Finally, from \eqref{theta0} and \eqref{theta3}, we obtain 
$$
\mu\left(\{x: {\rm Dens}(T_N(\Phi,x,\gamma_2))>\theta\}\right)>\theta,\quad \forall N\ge N_0.
$$
\end{proof}

Now we can give the proof of Proposition \ref{pro:ftoblock} as follows.

\begin{proof}[Proof of Proposition \ref{pro:ftoblock}]
For every $\varepsilon>0$, $x\in {\rm supp}(\mu)$, and $N\in \NN$, define
\begin{align*}
&\quad G_N(x,\varepsilon)=\Big\{n\in \mathbb{N}: ~\prod_{i=n-k}^{n-1} m(Df^N|_{E(f^{iN}x)})\ge {\rm e}^{(\chi_E^{-}-\varepsilon)kN},\\
&\quad \quad \quad \quad \quad \quad \quad \quad \quad  \quad \prod_{i=n-k}^{n-1} \|Df^N|_{F(f^{iN}x)}\|\le {\rm e}^{(\chi_F^{+}+\varepsilon)kN},\\
&\quad \quad \quad \quad \quad \quad \quad \quad \quad \quad \prod_{i=n}^{n+j-1} m(Df^N|_{E(f^{iN}x)})\ge {\rm e}^{(\chi_E^{-}-\varepsilon)jN},\\
&\quad \quad \quad \quad \quad \quad \quad \quad \quad \quad \prod_{i=n}^{n+j-1} \|Df^N|_{F(f^{iN}x)}\|\le {\rm e}^{(\chi_F^{+}+\varepsilon)jN}, \quad \forall 0\le k \le n,\quad \forall j\ge 1\Big\}.
\end{align*}
Apply Lemma \ref{lem:fpl} to the sub-additive sequences $\left\{\log\|Df^n|_{F}\|\right\}_{n\in \NN}$ and $\left\{-\log m(Df^n|_{F})\right\}_{n\in \NN}$. For any $\theta\in (0,1)$, there exists $N\in \NN$ such that 
$
\mu\left(\{x: {\rm Dens}(G_N(x,\varepsilon))>\theta\}\right)>\theta.
$
Hence, there exists an ergodic component $\nu$ of $\mu$ with respect to $f^N:=g$ such that ${\rm supp}(\nu)\subset {\rm supp}(\mu)$ and 
$$
\nu\left(\{x: {\rm Dens}(G_N(x,\varepsilon))>\theta\}\right)>\theta.
$$
Thus, we can choose a $\nu$-generic point $x_0\in {\rm supp}(\nu)$ satisfying
$$
{\rm Dens}(G_N(x_0,\varepsilon))>\theta.
$$
Observe that $H_{1}(\nu,N\varepsilon,x_0,g)=G_N(x_0,\varepsilon)$, so 
$$
{\rm Dens}(H_1(\nu,N\varepsilon,x_0,g))={\rm Dens}(G_N(x_0,\varepsilon))>\theta>0.
$$
The conclusion now follows by applying Lemma \ref{pro:blockoff} to  $g=f^N$.
%
\end{proof}

\section{Stable manifolds with sub-exponentially varying size}

This section is devoted to the proof of Theorem \ref{main:stable}. We begin in Subsection \ref{orsrt} by studying the iterated submanifolds along the orbit segment with resonance times on expansion/contraction and domination. This enables us to construct stable manifolds by accumulation for points in the resonance block under the hyperbolicity of the ergodic measure, which is done in Subsection \ref{stableblock}. The proof of Theorem \ref{main:stable} is then concluded in Subsection \ref{stableff}.

\subsection{Orbit segments with resonance times}\label{orsrt}

Let $g\in {\rm Diff}^1(M)$, $\nu$ be ergodic. Assume that $\U$ is an open neighborhood of ${\rm supp}(\nu)$, for which there exists a continuous invariant splitting 
$$
T_{\U}M=E\oplus F.
$$
Given $\theta>0$, the $F$-direction cone of width $\theta$ at a point $x\in \U$ is defined by
$$
\cC_{\theta}^F(x)=\{v=v_E+v_F\in E(x)\oplus F(x): \|v_E\|\le \theta \|v_F\|\}.
$$
Given a $C^1$-embedded submanifold $D$, let $d_D(x,y)$ be the distance from $x$ to $y$ along $D$.
A $C^1$-embedded submanifold $D$ of dimension ${\rm dim}(F)$ is called \emph{tangent to} $\cC_{\theta}^F$ if $T_xD\subset \cC_{\theta}^F(x)$ at every $x\in D$. 
The $E$-direction cone $\cC_{\theta}^E$ is defined analogously.

For any fixed $\varepsilon>0,$  denote $\eta_1=\eta_1(\varepsilon)=e^{\chi_F^+(\nu,g)-\chi_E^-(\nu,g)+2\varepsilon}$ and $\eta_2=\eta_1\cdot e^\varepsilon.$ The following result is a key step to obtain the stable manifold. It extends the previous work \cite[Lemma 3.7]{MCY18} to our present setting.

\begin{proposition}\label{pro:averf}
Let $g\in {\rm Diff}^1(M)$ preserving an  ergodic measure $\nu$. Assume that $\U$ is an open neighborhood of ${\rm supp}(\nu)$, for which there exists a continuous invariant splitting 
$
T_{\U}M=E\oplus F
$
with 
$$
\chi_F^+(\nu,g)<\min\left\{0, \chi_E^{-}(\nu,g)\right\}.
$$

For every sufficiently small $\varepsilon>0$, there exists $r=r(\varepsilon)>0$ such that for every $t\ge 1$, one can find $\theta_1=\theta_1(t)\in (0,1)$ with the following property: 
if the orbit segment $\{x,\dots,g^nx\}\subset {\rm supp}(\nu)$ satisfies
\begin{enumerate}[(i)]
\item\label{asu:avdo} 
$$
\prod_{i=n-k+1}^n \frac{\|Dg^{-1}|_{E(g^ix)}\|}{m(Dg^{-1}|_{F(g^ix)})}\le t^2\cdot { \eta_1}^{k},\quad \forall 1\le k \le n,
$$
\item\label{asu:avf}
$$
\prod_{i=1}^k m\left(Dg^{-1}|_{F(g^ix)}\right)\ge t^{-1}\cdot {\rm e}^{-\left(\chi_F^+(\nu,g)+\varepsilon\right)k},\quad \forall 1\le k \le n,
$$
\end{enumerate}
then for any disk $D$ containing $g^n(x)$, tangent to $\cC^F_{\theta_1}$ of radius greater than $r$ around $g^n(x)$, there exists a subdisk $\widetilde{D}\subset g^{-n}(D)$ containing $x$ such that for every $0\le k \le n$, it holds that 
\begin{enumerate}[(1)]
\item\label{item:sub} $g^k(\widetilde{D})\subset B(g^kx,r)$, and tangent to $\cC_{t^2\cdot {\eta_{2}}^{n-k}\cdot \theta_1}^F$.
\item\label{item:radius} $\widetilde{D}$ is a disk of radius $r\cdot t^{-1}$ around $x$.
\item \label{item:disest} $d(g^kz_1,g^kz_2)\le t\cdot {\rm e}^{(\chi_F^+(\nu,g)+2\varepsilon)k}d(z_1,z_2),\quad \forall z_1,z_2\in \widetilde{D}.$
\end{enumerate}

\end{proposition}
Similar to Proposition \ref{pro:averf}, one can also obtain:
\begin{proposition}\label{pro:avere}
Let $g\in {\rm Diff}^1(M)$ and $\nu$ be ergodic. Assume that $\U$ is an open neighborhood of ${\rm supp}(\nu)$, for which there exists a continuous invariant splitting 
$
T_{\U}M=E\oplus F
$
with 
$$
\chi_E^-(\nu,g)>\max\left\{0, \chi_F^{+}(\nu,g)\right\}.
$$

For every sufficiently small $\varepsilon>0$, there exists $r=r(\varepsilon)>0$ such that for every $t\ge 1$, one can find $\theta_1=\theta_1(t)\in (0,1)$ with the following property:
if the orbit segment $\{x,\dots,g^nx\}\subset {\rm supp}(\nu)$ satisfies
\begin{enumerate}[(i)]
\item 
$$
\prod_{i=0}^{k-1}\frac{\|Dg|_{F(g^ix)}\|}{m\left(Dg|_{E(g^ix)}\right)}\le t^2\cdot {\rm \eta_1}^{k},\quad \forall 1\le k \le n,
$$
\item
$$
\prod_{i=n-k}^{n-1} m\left(Dg|_{E(g^ix)}\right)\ge t^{-1}\cdot {\rm e}^{(\chi_E^-(\nu,g)-\varepsilon)k},\quad \forall 1\le k \le n,
$$
\end{enumerate}
then for any disk $D$ containing $x$, tangent to $\cC^E_{\theta_1}$ of radius greater than $r$ around $x$, there exists a subdisk $\widetilde{D}\subset g^{n}(D)$ containing $x$ such that for every $0\le k \le n$, it holds that 
\begin{enumerate}[(1)]
\item  $g^{-k}(\widetilde{D})\subset B(g^{n-k}x,r)$, and tangent to $\cC_{t^2\cdot {\eta_2}^k\cdot \theta_1}^E$.
\item  $\widetilde{D}$ is a disk of radius $r\cdot t^{-1}$ around $g^n(x)$.
\item   $d(g^{-k}z_1,g^{-k}z_2)\le t\cdot {\rm e}^{-(\chi_E^-(\nu,g)-2\varepsilon)k}d(z_1,z_2),\quad \forall z_1,z_2\in \widetilde{D}.$
\end{enumerate}
\end{proposition}

To prove Proposition \ref{pro:averf}, we show the following lemma first.
\begin{lemma}\label{lem:averf}
Let $g\in {\rm Diff}^1(M)$ and $\nu$ be ergodic. Assume that $\U$ is an open neighborhood of ${\rm supp}(\nu)$, for which there exists a continuous invariant splitting 
$
T_{\U}M=E\oplus F
$
with 
$$
\chi_F^+(\nu,g)<\chi_E^{-}(\nu,g).
$$

For every sufficiently small $\varepsilon>0$, there exists $r=r(\varepsilon)>0$ such that for every $\theta>0$ and $t\ge 1$, if the orbit segment $\{x,\dots,g^nx\}\subset {\rm supp}(\nu)$ satisfies
\begin{equation}\label{assumption dominated}
\prod_{i=n-k+1}^n \frac{\|Dg^{-1}|_{E(g^ix)}\|}{m(Dg^{-1}|_{F(g^ix)})}\le t^2\cdot {\eta_1}^k,\quad \forall 1\le k \le n,
\end{equation}
then for any submanifold $D$ tangent to $\cC_{\theta}^F$, satisfying 
$$
g^{-n+k}(D)\subset B(g^{k}x,r),\quad \forall 0\le k \le n,
$$
one has that $g^{-k}(D)$ is tangent to $\cC^F_{t^2\cdot {\eta_2}^k\cdot \theta}$ for every $1\le k \le n$.
\end{lemma}

\begin{proof}
Choose $\varepsilon$ small enough such that 
$$
0<\varepsilon<\frac{1}{10}\left(\chi_E^{-}(\nu,g)-\chi_F^{+}(\nu,g)\right).
$$
Choose $r_0>0$ such that $\overline{B}({\rm supp}(\nu),r_0)\subset \U$. 
By the uniform continuity of $Dg$ and the subbundles on $\overline{B}({\rm supp}(\nu),r_0)$, there exists a positive real number $r=r(\varepsilon)<r_0$ such that for any points $z_1, z_2\in \overline{B}({\rm supp}(\nu),r_0)$ with $d(z_1,z_2)\le r$, we have 
\begin{equation}\label{estaver1}
\frac{\|Dg^{-1}|_{E(z_1)}\|}{\|Dg^{-1}|_{E(z_2)}\|}, \frac{m(Dg^{-1}|_{F(z_1)})}{m(Dg^{-1}|_{F(z_2)})}\in \left({\rm e}^{-\varepsilon/4}, {\rm e}^{\varepsilon/4}\right).
\end{equation}
For any $y\in g^{-n}D$, we know by construction that $d(g^kx,g^ky)\le r$ for every $1\le k \le n$. Using assumption \eqref{assumption dominated} and estimate \eqref{estaver1}, we obtain
\begin{equation}\label{eq:coneest11}
\prod_{i=n-k+1}^n\frac{\|Dg^{-1}|_{E(g^iy)}\|}{m\left(Dg^{-1}|_{F(g^iy)}\right)}\le \prod_{i=n-k+1}^n\frac{{\rm e}^{\varepsilon/4}\cdot\|Dg^{-1}|_{E(g^ix)}\|}{{\rm e}^{-\varepsilon/4}\cdot m\left(Dg^{-1}|_{F(g^ix)}\right)}
\le t^2\cdot \eta_2^k.
\end{equation}
Given $1\le k \le n$, take any vector $v\in T_{g^{n-k}y}(g^{-k}D)$, it suffices to prove $v\in \cC^F_{t^2\cdot \eta_2^k\cdot \theta}$. Write $v=v_E+v_F$ with $v_E\in E(g^{n-k}y)$, $v_F\in F(g^{n-k}y)$. By the invariance of the splitting, 
$$
Dg_{g^{n-k}y}^k(v)=Dg_{g^{n-k}y}^k(v_E)+Dg_{g^{n-k}y}^{k}(v_F)\in T_{g^ny}D,
$$
with $Dg_{g^{n-k}y}^k(v_E)\in E(g^ny)$ and $Dg_{g^{n-k}y}^k(v_F)\in F(g^ny)$.
Since $D$ is tangent to $\cC_{\theta}^F$, with the estimate \eqref{eq:coneest11} we conclude
$$
\frac{\|v_E\|}{\|v_F\|}\le \left(\prod_{i=n-k+1}^n\frac{\|Dg^{-1}|_{E(g^iy)}\|}{m\left(Dg^{-1}|_{F(g^iy)}\right)}\right)\cdot \frac{\|Dg_{g^{n-k}y}^k(v_E)\|}{\|Dg_{g^{n-k}y}^k(v_F)\|}
\le t^2\cdot \eta_2^k\cdot \theta.
$$
This completes the proof. 
\end{proof}

By applying Lemma \ref{lem:averf}, we can give the proof of Proposition \ref{pro:averf}.

\begin{proof}[Proof of Proposition \ref{pro:averf}]

By the assumption on Lyapunov exponents, we can choose $\varepsilon$ small enough such that 
$$
0<\varepsilon<\min\left\{-\frac{1}{10}\chi_F^{+}(\nu,g),\frac{1}{10}\left(\chi_E^{-}(\nu,g)-\chi_F^{+}(\nu,g)\right)\right\}.
$$
Let
$$
\zeta_1={\rm e}^{\chi_F^+(\nu,g)+\varepsilon}.
$$
Choose $r_0>0$ and $\theta_0>0$ to satisfy $\overline{B}({\rm supp}(\nu),r_0)\footnote{For $A\subset M$ and $\rho>0$, denote $\overline{B}(A,\rho)=\{y:d(y,A)\le \rho\}$. }\subset \U$ and
\begin{equation}\label{diseq}
{\rm e}^{-\varepsilon/4}d(z_1,z_2)\le d_W(z_1,z_2) \le {\rm e}^{\varepsilon/4}d(z_1,z_2),\quad \forall z_1,z_2\in W,
\end{equation}
whence $W\subset \U$ is tangent to $\cC_{\theta_0}^F$ with ${\rm diam}(W)\le 2r_0$. 
Let $r=r(\varepsilon)<r_0$ be the constant from Lemma \ref{lem:averf} satisfying \eqref{estaver1}.
Moreover, for the given $t\ge 1$, we can choose $\theta=\theta_1(t)\le \theta_0$ small enough such that for any subbundle $\widetilde{F}(y)\subset \cC_{t^2\cdot \theta_1}^F(y)$ with $y\in \overline{B}({\rm supp}(\nu),r_0)$, it holds that 
\begin{equation}\label{eq:nearby1}
\frac{m\left(Dg^{-1}|_{\widetilde{F}(y)}\right)}{m\left(Dg^{-1}|_{F(y)}\right)}>{\rm e}^{-\varepsilon/4}.
\end{equation}

Let $D$ be a smooth disk tangent to $\cC_{\theta_1}^F$, satisfying $d_D(g^nx,\partial D)>r$. Denote by $B_n$ the connected component of $D\cap \overline{B}(g^nx,r)$ containing $g^n(x)$. Define inductively the sets
$$
B_k=g^{-1}(B_{k+1})\cap \overline{B}(g^kx,r), \quad \forall 0\le k \le n-1.
$$
Thus,
\begin{equation}\label{containin}
g^k(B_0)\subset B_k\subset \overline{B}(g^kx,r),\quad \forall 0\le k \le n
\end{equation}
By Lemma \ref{lem:averf},
$g^{n-k}(B_0)$ is tangent to $\cC_{t^2\cdot \eta_2^k\cdot \theta_1}^F\subset \cC_{t^2\cdot \theta_1}^F$ for every $1\le k \le n$.

Now we show that $B_0$ exhibits radius no less than $t^{-1}\cdot r$ around $x$. By contradiction, suppose that there is $z\in \partial B_0$ such that $d_{B_0}(x,z)<t^{-1}\cdot r$. Then \eqref{containin} ensures that
$
g^i(z)\in g^i(B_0)\subset B_i\subset \overline{B}(g^ix,r),
$
for every $1\le i \le n$.
For each $1\le i \le n$ we can choose a point $x_i\in g^i(B_0)$ such that 
$$
d_{g^{i-1}(B_0)}(g^{i-1}y,g^{i-1}z)\ge m\left(Dg^{-1}|_{T_{x_i}(g^iB_0)}\right)\cdot d_{g^i(B_0)}(g^iy,g^iz).
$$
Since  $d(x_i,g^ix)\le r$ from the fact $g^i(B_0)\subset \overline{B}(g^ix,r)$, together with \eqref{estaver1}, \eqref{eq:nearby1} and the assumption \eqref{asu:avf}, it follows that 
\begin{align}\label{zhongzhi}
d_{B_0}(x,z)&\ge \left(\prod_{i=1}^k m\left(Dg^{-1}|_{T_{x_i}(g^iB_0)}\right)\right)\cdot d_{g^kB_0}(g^kx,g^kz)\notag\\
&\ge \left( \prod_{i=1}^k {\rm e}^{-\varepsilon/2}\cdot m\left(Dg^{-1}|_{F(g^ix)}\right)\right)\cdot d_{g^kB_0}(g^kx,g^kz)\notag\\
&\ge t^{-1}(\zeta_1\cdot {\rm e}^{\varepsilon/2})^{-k}\cdot d_{g^kB_0}(g^kx,g^kz).
\end{align}
This together with the assumption that $d_{B_0}(x,z)<t^{-1}\cdot r$ yields
$$
d_{g^kB_0}(g^kx,g^kz)<r\cdot (\zeta_1\cdot {\rm e}^{\varepsilon/2})^k< r,\quad 1\le k \le n.
$$
Thus, using the facts 
$$
z\in \partial B_0, \quad d_{B_0}(x,z)<r,\quad d_{g(B_0)}(gx,gz)<r,
$$
we conclude that $g(z)\in \partial B_1$. Inductively, we obtain $g^n(z)\in 
\partial B_n$ and $d_{B_n}(g^nx,g^nz)<r$, which  contradicts  the fact $d_{B_n}(g^nx,\partial B_n)\ge r$.

Take
$$
\widetilde{D}=\left\{z\in B_0: d_{B_0}(x,z)\le t^{-1}\cdot r\right\}.
$$
Then $\widetilde{D}$ satisfies the desired properties \eqref{item:sub} and \eqref{item:radius}.
Now we show \eqref{item:disest}. For every $z_1,z_2\in \widetilde{D}$, for every $1\le k \le n$ and $1\le i \le k$, similar to the argument in \eqref{zhongzhi} we get
\begin{align*}
d_{\widetilde{D}}(z_1,z_2)&\ge\left(\prod_{i=1}^{k}{\rm e}^{-\varepsilon/2}\cdot m\left(Dg^{-1}|_{F(g^ix)}\right)\right)\cdot d_{g^k\widetilde{D}}(g^kz_1,g^kz_2)\\
&\ge t^{-1}(\zeta_1\cdot {\rm e}^{\varepsilon/2})^{-k}\cdot d_{g^k\widetilde{D}}(g^kz_1,g^kz_2).
\end{align*}
This together with \eqref{diseq}, this yields
$$
d(g^kz_1,g^kz_2)\le t\cdot (\zeta_1\cdot{\rm e}^{\varepsilon/2})^k\cdot {\rm e}^{\varepsilon/2}d(z_1,z_2)\le t\cdot {\rm e}^{(\chi_F^+(\nu,g)+2\varepsilon)k}d(z_1,z_2).
$$
\end{proof}

\subsection{Stable manifolds on resonance blocks}\label{stableblock}

We establish the stable manifolds for points in the resonance block w.r.t. the hyperbolic ergodic measure.

\begin{theorem}\label{thm:stable-block}
Let $g\in {\rm Diff}^1(M)$ preserving an  ergodic measure $\nu$. Assume that $\U$ is an open neighborhood of ${\rm supp}(\nu)$, for which there exists a continuous invariant splitting 
$
T_{\U}M=E\oplus F
$
with 
$$
\chi_F^+(\nu,g)<\min\left\{0, \chi_E^{-}(\nu,g)\right\}.
$$

For any sufficiently small $\varepsilon>0$, there exists $r=r(\varepsilon)>0$ such that for any $t\ge 1$, if $\Lambda_t^{\varepsilon}$ is a resonance block w.r.t. $(g,\nu,E\oplus F)$, then for every $x\in \Lambda_t^{\varepsilon}$, there exists a submanifold $W^F_{\rm loc}(x,g)$ satisfying
\begin{enumerate}[(1)]
\item $W^F_{\rm loc}(x,g)$ is an embedded $C^1$-disk centered at $x$ of radius $r\cdot t^{-1}$, and tangent to $F$.
\item For every $y,z\in W_{\rm loc}^F(x,g)$, 
$$
d(g^ky,g^kz)\le t \cdot {\rm e}^{(\chi_F^+(\nu,g)+2\varepsilon)k}\cdot d(y,z),\quad \forall k\ge 1.
$$
\end{enumerate}
\end{theorem}

\begin{proof}
Writing for simplicity that 
$\chi_E^+=\chi_E^+(\nu,g)$and $\chi_F^{+}=\chi_F^{+}(\nu,g).$
Take any $\varepsilon>0$ sufficiently small such that 
\begin{equation}\label{epsilon-g}
0<\varepsilon<\min\left\{-\frac{1}{10}\chi_F^{+},\frac{1}{10}(\chi_E^{-}-\chi_F^{+})\right\}.
\end{equation}
For any $t\ge 1$, we fix $\Lambda_{t}^{\varepsilon}=\Lambda_{t}^{\varepsilon}(\nu,x_0,g)$ for some $x_0\in {\rm supp}(\nu)$.

By definition, for every $x\in \Lambda_t^{\varepsilon}$ there exists a subsequence $\{n_k\}_{k\ge 1}\subset H_t(\nu,\varepsilon,x_0,g)$ such that $g^{n_k}(x_0)\to x$. For every $j\ge 1$, pick $\ell_j$ so that $n_{\ell_j}-n_j\ge j$. Since $n_j, n_{\ell_j}\in H_t(\nu,\varepsilon,x_0,g)$, we have the following two estimates:
\begin{itemize}
\item[---] For every $1\le k\le n_{\ell_j}$, 
\begin{align*}
\prod_{i=n_{\ell_j}-k+1}^{n_{\ell_j}}\frac{\left\|Dg^{-1}|_{E(g^ix_0)}\right\|}{m\left(Dg^{-1}|_{F(g^ix_0)}\right)}&=\prod_{i=n_{\ell_j}-k+1}^{n_{\ell_j}}\frac{m\left(Dg|_{E(g^{i-1}x_0)}\right)^{-1}}{\|Dg|_{F(g^{i-1}x_0)}\|^{-1}}\\
&=\prod_{i=n_{\ell_j}-k+1}^{n_{\ell_j}}\frac{\|Dg|_{F(g^{i-1}x_0)}\|}{m\left(Dg|_{E(g^{i-1}x_0)}\right)}\\
&\le t^2\cdot {\rm e}^{(\chi_F^+-\chi_E^++2\varepsilon)k}.
\end{align*}
\item[---] For every $k\ge 1$,
\begin{align*}
\prod_{i=1}^k m\left(Dg^{-1}|_{F(g^i\circ g^{n_j}(x_0))}\right)&=\prod_{i=1}^k\left\|Dg|_{F(g^{i-1}\circ g^{n_j}(x_0))}\right\|^{-1}\\
&\ge t^{-1}\cdot {\rm e}^{-(\chi_F^++\varepsilon)k}.
\end{align*}
\end{itemize}
By Proposition \ref{pro:averf}, there exist $r=r(\varepsilon)<r_0$ and $\theta_1=\theta_1(t)$ with the following property: There exists a $C^1$-disk $D_j$ containing $g^{n_j}(x_0)$ such that 
\begin{enumerate}[(1)]
\item\label{saver} $D_j\subset B(g^{n_j}x_0,r)$ has radius $r\cdot t^{-1}$ around $g^{n_j}(x_0)$.
\item\label{taver} $D_j$ is tangent to $\cC_{t^2\cdot {\eta_2}^{j}\cdot \theta_1}^F$, where $\eta_{2}=e^{\chi_F^+-\chi_E^-+3\varepsilon}$.
\item\label{fcaver} $d(g^ky, g^kz)\le t\cdot {\rm e}^{(\chi_F^++2\varepsilon)k} \cdot d(y,z),\quad \forall z_1,z_2\in D_j,\quad \forall 1\le k \le n_{\ell_j}-n_j.$
\end{enumerate}

By applying \eqref{taver} and the uniform continuity of $F$, one deduces that $\{D_j\}_{j\ge 1}$ is uniformly bounded and uniformly equicontinuous in the $C^1$-topology (see \cite[Claim 2]{CMZ25} for details).
%
By Arzela-Ascoli theorem, up to considering subsequences we can assume that 
$D_j$ converges to a $C^1$ disk $D(x)$ in $C^1$-topology. 

By \eqref{saver}, $D(x)$ admits radius $r\cdot t^{-1}$ around $x$.
By \eqref{taver}, one knows that the angle between $D_j$ and $F$ goes to zero as $j\to \infty$, this together with the uniform continuity of $F$ yields that $D(x)$ is tangent to $F$. 
 For any $y,z\in D(x)$, due to the convergence of $D_j\to D(x)$ there exist $y_j, z_j\in D_j$ such that $y_j\to y$ and $z_j\to z$ as $j\to \infty$, as $D_j\to D(x)$. Note that $n_{l_j}-n_j\ge j.$ By \eqref{fcaver}, we obtain
$$
d(g^k(y_j),g^k(z_j))\le t\cdot {\rm e}^{(\chi_F^++2\varepsilon)k} d(y_j,z_j),\quad \forall 1\le k \le j,\quad \forall j\ge 1.
$$
Letting $j\to \infty$, we deduce that
$$
d(g^k y, g^k z)\le t\cdot {\rm e}^{(\chi_F^+(\nu,g)+2\varepsilon)k}\cdot d(y,z),\quad \forall k\ge 1.
$$
To complete the proof, it suffices to take $W_{{\rm loc}}^F(x,g)=D(x)$.
\end{proof}

The following proposition will be  used in the proof of Theorem \ref{main:stable}.
\begin{proposition}\label{pro:forstableff}
 Under the assumption of Theorem \ref{thm:stable-block}, assume further that $\nu(\Lambda_t^{\varepsilon})>0$. Then for any $\tau<\min\{0,\chi_E^--2\varepsilon\}$ and for $\nu$-almost every $x\in \Lambda_t^{\varepsilon}$, if a subset $B\subset M$ satisfies  $x\in B$ and ${\rm e}^{-\tau n}{\rm diam}(g^n(B))\to 0$ as $n\to \infty$, then there exists $n_K\in \NN$ large enough such that 
 $$
 g^{n_K}(B)\subset W^F_{{\rm loc}}(g^{n_K}(x),g).
 $$
\end{proposition}

\begin{proof}
By Poincar\'e recurrence theorem, for $\nu$-almost every $x\in \Lambda_t^{\varepsilon}$, one can find $\{n_k\}_{k\ge 1}\subset \NN$ such that $g^{n_k}(x)\in \Lambda_t^{\varepsilon}$. For any $y\in B$, since $\tau<0$, we have for every $n\in \NN$ that
$$
d(g^nx,g^ny)<{\rm e}^{-\tau n}\cdot d(g^nx,g^ny)\le {\rm e}^{-\tau n}\cdot {\rm diam}(g^n(B)).
$$
By the assumption on $B$, it follows that $\lim_{n\to \infty}d(g^nx,g^ny)=0$. For $r=r(\varepsilon)$ given in Theorem \ref{thm:stable-block},  choose $n_K\in\{n_k\}_{k\ge 1}$  such that 
$$
d(g^nx,g^ny)<r/2,\quad \forall n\ge n_K.
$$
We may also assume $n_K$ is large enough such that if $D_0$ is a $C^1$-disk  tangent to $\cC^E_{t^2\cdot\theta_1}$ of radius $\frac{1}{2}r\cdot t^{-1}$ centered at $g^{n_K}(y)$, then it has a transverse intersection with $W^F_{\rm loc}(g^{n_K}(x),g)$ at some point $g^{n_K}(z)$.
Let $D_{\ell}=g(D_{\ell-1})\cap B(g^{n_K+\ell}y, r/2)$ for every $\ell\in \NN$. Then $D_{\ell}\subset B(g^{n_K+\ell}x, r)$.
As $g^{n_K}(x)\in \Lambda_t^{\varepsilon}$, by Corollary \ref{cor:lambda-bl} and Proposition \ref{pro:avere}, we know $D_{\ell}$ is tangent to $\cC^E_{t^2\cdot \theta_1}$ for every $\ell\in \NN$. Moreover, for every $j>K$, 
\begin{equation}\label{add1}
d(g^{n_K}y, g^{n_K}z)\le t \cdot {\rm e}^{-(\chi_E^--2\varepsilon)(n_j-n_K)}d(g^{n_j}y,g^{n_j}z).
\end{equation}
On the other hand, as $g^{n_K}z\in W^F_{\rm loc}(g^{n_K}x,g)$, we have 
\begin{equation}\label{add2}
d(g^{n_j}x, g^{n_j}z)\le t\cdot {\rm e}^{(\chi_F^++2\varepsilon)(n_j-n_K)}\cdot d(g^{n_K}x,g^{n_K}z).
\end{equation}
Combining \eqref{add1} and \eqref{add2}, we obtain
\begin{align*}
d(g^{n_K}y, f^{n_K}z)&\le t\cdot {\rm e}^{-(\chi_E^--2\varepsilon)(n_j-n_K)}\left(d(g^{n_j}y,g^{n_j}x)+d(g^{n_j}x,g^{n_j}z)\right)\\
&\le t\cdot {\rm e}^{-(\chi_E^--2\varepsilon)(n_j-n_K)}\cdot {\rm diam}(g^{n_j}(B))\\
&+t^2\cdot {\rm e}^{(\chi_F^+-\chi_E^-+4\varepsilon)(n_j-n_K)}\cdot d(g^{n_K}x,g^{n_K}z).
\end{align*}
It follows from $\tau<\chi_E^--2\varepsilon$ that
$
t\cdot {\rm e}^{-(\chi_E^--2\varepsilon)(n_j-n_K)}\cdot {\rm diam}(g^{n_j}(B))\to 0
$
as $j\to +\infty$. Recall the choice of $\varepsilon$ in \eqref{epsilon-g}, we know  
$$
t^2\cdot{\rm e}^{(\chi_F^+-\chi_E^-+4\varepsilon)(n_j-n_K)}\cdot d(g^{n_K}x,g^{n_K}z)\to 0,\quad \textrm{as}~j\to \infty.
$$ 
As a result, $g^{n_K}y=g^{n_K}z$, thus $g^{n_K}y\in W^F_{{\rm loc}}(g^{n_K}x,g)$.
We conclude the result by the arbitrariness of $y\in B$.
\end{proof}

Following the argument as in Theorem \ref{thm:stable-block}, using Proposition \ref{pro:avere} instead of Proposition \ref{pro:averf}, we get the next result. 

\begin{theorem}\label{thm:unstable-block}
Let $g\in {\rm Diff}^1(M)$ preserving an  ergodic measure $\nu$. Assume that $\U$ is an open neighborhood of ${\rm supp}(\nu)$, for which there exists a continuous invariant splitting 
$
T_{\U}M=E\oplus F
$
with 
$$
\chi_E^-(\nu,g)>\max\left\{0, \chi_F^{+}(\nu,g)\right\}.
$$

For any sufficiently small $\varepsilon>0$, there exists $r>0$ such that for any $t\ge 1$, if $\Lambda_t^{\varepsilon}$ is a resonance block w.r.t. $(g,\nu,E\oplus F)$, then for every $x\in \Lambda_t^{\varepsilon}$, there exists a submanifold $W^E_{{\rm loc}}(x,g)$ satisfying
\begin{enumerate}[(1)]
\item $W^E_{\rm loc}(x,g)$ is an embedded $C^1$-disk centered at $x$ of radius $r\cdot t^{-1}$, and tangent to $E$.
\item For every $y,z\in W_{\rm loc}^E(x,g)$, 
$$
d(g^{-k}y,g^{-k}z)\le t \cdot {\rm e}^{-(\chi_E^-(\nu,g)-2\varepsilon)k}\cdot d(y,z),\quad \forall k\ge 1.
$$
\end{enumerate}
\end{theorem}

As a consequence of Theorem \ref{thm:stable-block} and Theorem \ref{thm:unstable-block}, we obtain

\begin{theorem}\label{thm:un-stableblock}
Let $g\in {\rm Diff}^1(M)$ and $\nu$ be ergodic. Assume that $\U$ is an open neighborhood of ${\rm supp}(\mu)$, for which there exists a continuous invariant splitting 
$
T_{\U}M=E\oplus F
$
with 
$$
\chi_E^-(\nu,g)>0>\chi_F^{+}(\nu,g).
$$

For any sufficiently small $\varepsilon>0$, there exists $r>0$ such that for any $t\ge 1$, if $\Lambda_t^{\varepsilon}$ is a resonance block w.r.t. $(g,\nu,E\oplus F)$, then for every $x\in \Lambda_t^{\varepsilon}$, there exists a submanifold $W^E_{\rm loc}(x,g)$ and $W^F_{{\rm loc}}(x,g)$ satisfying
\begin{enumerate}[(1)]
\item $W^*_{\rm loc}(x,g)$ is an embedded $C^1$-disk centered at $x$ of radius $r\cdot t^{-1}$, and tangent to $*\in E, F$.
\item For every $y,z\in W_{\rm loc}^E(x,g)$, 
$$
d(g^{-k}y,g^{-k}z)\le t \cdot {\rm e}^{-(\chi_E^-(\nu,g)-2\varepsilon)k}\cdot d(y,z),\quad \forall k\ge 1.
$$
\item For every $y,z\in W_{\rm loc}^F(x,g)$, 
$$
d(g^{k}y,g^{k}z)\le t \cdot {\rm e}^{(\chi_F^+(\nu,g)+2\varepsilon)k}\cdot d(y,z),\quad \forall k\ge 1.
$$
\end{enumerate}
\end{theorem}

%

\subsection{Proof of Theorem \ref{main:stable}}\label{stableff}
%
%
%

Denote 
$$
m(Df)=\inf_{x\in M}m(Df_x), \quad \|Df\|=\sup_{x\in M}\|Df_x\|,\quad \|Df^{-1}\|=\sup_{x\in M}\|Df^{-1}_x\|.
$$
By Proposition \ref{pro:ftoblock}, for any $\varepsilon>0$ sufficiently small, there exists $N\in \NN$, and an ergodic component $\nu$ of $\mu$ w.r.t. $f^N$ and a $\nu$-generic point $x_0\in {\rm supp}(\nu)$ such that 
\begin{equation}\label{fpfull}
\nu\left(\bigcup_{t\ge 1}\Lambda_t^{\frac{N\varepsilon}{2}}\left(\nu,x_0,f^N\right)\right)=1.
\end{equation}
Set for simplicity that
$$
g=f^N,\quad
\Lambda_t^{\frac{N\varepsilon}{2}}=\Lambda_t^{\frac{N\varepsilon}{2}}\left(\nu,x_0,g\right), \quad 
R=\bigcup_{t\ge 1}\Lambda_t^{\frac{N\varepsilon}{2}}.
$$
Note also that $\mu$ can be decomposed as follows
$$
\mu=\frac{1}{m}\sum_{i=0}^{m-1}f_{\ast}^i\nu\quad \textrm{for some}~~m\big|N.
$$
Without loss of generality, we assume $m=N$. Then, there exist pairwise disjoint measurable subsets $P_0,\dots, P_{N-1}$ such that $f_{\ast}^i\nu(P_i)=1$ and $f(P_i)=P_{i+1}$ for each $i$ $(\textrm{mod}~ N)$.
Since $\nu(R)=1$ by \eqref{fpfull}, it follows from Lemma \ref{lem:temp} that there exists a $g$-invariant full $\nu$-measure subset ${R}^{\nu}$ of $R\cap P_0$ such that 
\begin{equation}\label{temt}
	\lim_{k\to \pm\infty}\frac{1}{|k|}\log T(f^{kN}y)=0,\quad \forall y\in R^{\nu}.
\end{equation}
Let 
$$
\cR=\cR(\varepsilon)=\bigcup_{i=0}^{N-1}f^i(R^{\nu}).
$$
Thus, we have $\mu(\cR)=1$ by definition. 
It suffices to show the existence of stable manifolds of $f$ at $x\in \cR$.

By the assumption on Lyapunov exponents, we have
$$
\chi_F^+(\nu,g)<\min\left\{0, \chi_E^{-}(\nu,g)\right\},
$$
using the facts $\chi_F^+(\mu,g)=N\cdot \chi_F^+(\nu,f)$ and $\chi_E^-(\mu,g)=N\cdot \chi_E^-(\nu,f)$. Thus, one can apply Theorem \ref{thm:stable-block} to $g$, $\nu$ and $\frac{N\varepsilon}{2}$. Hence, there exists $r>0$ such that for any $t\ge 1$, each $x\in \Lambda_t^{\frac{N\varepsilon}{2}}$ admits a local stable manifold $W^F_{\rm loc}(x,g)$ of size $r\cdot t^{-1}$. Moreover, for every $y,z\in W_{\rm loc}^F(x,g)$, 
\begin{equation}\label{gstable}
d(g^ky,g^kz)\le t \cdot {\rm e}^{(\chi_F^+(\nu,g)+N\cdot\varepsilon)k}\cdot d(y,z),\quad \forall k\ge 1.
\end{equation}
Since $\{P_i\}_{0\le i \le N-1}$ are pairwise disjoint, for each $x\in \cR$, there exists a unique $0\le \kappa(x)\le N-1$ such that $f^{-\kappa(x)}(x)\in R\cap P_0$.
Thus, we can define
$$
\delta(x)=m(Df)^N\cdot r\cdot {\rm e}^{-\varepsilon}\cdot T(f^{-\kappa(x)}(x))^{-1},\quad \forall x\in \cR,
$$
and
$$
A(x)=\|Df\|^{2N}\cdot \|Df^{-1}\|^N\cdot {\rm e}^{\varepsilon}\cdot {\rm e}^{-(\chi_F^+(\mu,f)+\varepsilon)N}\cdot T(f^{-\kappa(x)}(x)),\quad \forall x\in \cR.
$$
\smallskip

For every $x\in \cR$, note that $f^{-\kappa(x)}(x)\in \Lambda_t^{\frac{N\varepsilon}{2}}$ for $t=T(f^{-\kappa(x)}(x))\cdot {\rm e}^{\varepsilon}$ by the definition of $T$. Hence,
$$
{\rm radius}\left(W^F_{\rm loc}(f^{-\kappa(x)}(x),g)\right)\ge r\cdot t^{-1}= r\cdot \left(T(f^{-\kappa(x)}(x))\cdot {\rm e}^{\varepsilon}\right)^{-1}.
$$
As a result, we obtain
\begin{align*}
{\rm radius}\left(f^{\kappa(x)}(W^F_{\rm loc}(f^{-\kappa(x)}(x),g))\right)&\ge m(Df^{\kappa(x)})\cdot {\rm radius}\left(W^F_{\rm loc}(f^{-\kappa(x)}(x),g))\right)\\
&\ge m(Df)^N\cdot r\cdot \left(T(f^{-\kappa(x)}(x))\cdot {\rm e}^{\varepsilon}\right)^{-1}\\
&= \delta(x).
\end{align*}
Define
\begin{equation}\label{def:staf}
W^F_{\rm loc}(x,f)=f^{\kappa(x)}(W^F_{\rm loc}(f^{-\kappa(x)}(x),g)).
\end{equation}
Then it contains a $C^1$-disk centered at $x$ of radius $\delta(x)$. This verifies Item \eqref{sidf} of Theorem \ref{main:stable}.

\medskip
Next, we show Item \eqref{contsf} of Theorem \ref{main:stable}. Fix any $x\in \cR$.
For any $y,z\in W^F_{\rm loc}(x,f)$,  since $W^F_{\rm loc}(x,f)= f^{\kappa(x)}\left(W^F_{\rm loc}(f^{-\kappa(x)}(x),g)\right)$, we have $f^{-\kappa(x)}y, f^{-\kappa(x)}z\in W^F_{\rm loc}(f^{-\kappa(x)}(x),g)$.  For any $n\ge 1$, let $n=kN+\ell$ for some $k\ge 0$ and $0\le \ell\le N-1$. Note that $f^{-\kappa(x)}(x)\in \Lambda_t^{\frac{N\varepsilon}{2}}$ for $t=T(f^{-\kappa(x)}(x))\cdot {\rm e}^{\varepsilon}$. By applying \eqref{gstable}, we obtain
\begin{align*}
&d(f^ny,f^nz)\\
&=d\left(f^{kN+\ell+\kappa(x)}\circ f^{-\kappa(x)}(y), f^{kN+\ell+\kappa(x)}\circ f^{-\kappa(x)}(z)\right)\\
&=d\left(f^{\ell+\kappa(x)}\circ g^k(f^{-\kappa(x)}(y)), f^{\ell+\kappa(x)}\circ g^k(f^{-\kappa(x)}(z))\right)\\
&\le \|Df\|^{2N}\cdot d\left(g^k(f^{-\kappa(x)}(y)), g^k(f^{-\kappa(x)}(z))\right)\\
&\le \|Df\|^{2N}\cdot t\cdot {\rm e}^{(\chi_F^+(\nu,g)+N\cdot \varepsilon)k}\cdot d(f^{-\kappa(x)}(y),f^{-\kappa(x)}(z))\\
&\le \|Df\|^{2N}\cdot T(f^{-\kappa(x)}(x))\cdot {\rm e}^{\varepsilon}\cdot {\rm e}^{(\chi_F^+(\mu,f)+ \varepsilon)kN}\cdot \|Df^{-1}\|^N\cdot d(y,z)\\
& = \|Df\|^{2N}\cdot \|Df^{-1}\|^N \cdot T(f^{-\kappa(x)}(x))\cdot {\rm e}^{\varepsilon}\cdot {\rm e}^{-(\chi_F^+(\mu,f)+\varepsilon)\ell}\cdot {\rm e}^{\left(\chi_F^+(\mu,f)+\varepsilon\right)n}d(y,z)\\
&\le A(x)\cdot {\rm e}^{(\chi_F^+(\mu,f)+\varepsilon)n}\cdot d(y,z).
\end{align*}

%

Now we check \eqref{subexp}. We just show the sub-exponentiality of the sequence$\{\delta(f^nx)\}_{n\in \ZZ}$. The result on $\{A(f^nx)\}_{n\in \ZZ}$ can be deduced with the same argument. 
We first show the convergence when $n\to +\infty$.
Given any $x\in \cR$, for every $-\kappa(x)\le \ell <N-\kappa(x)$ and every $k\ge 1$, we have
$$
f^{-(\kappa(x)+\ell)}\circ f^{kN+\ell}(x)=f^{kN}\circ f^{-\kappa(x)}(x)\in R^{\nu}.
$$
%
Observe that $0\le \kappa(x)+\ell< N$, with the uniqueness of $\kappa$ we deduce from above that
$\kappa(f^{kN+\ell})=\kappa(x)+\ell$. Consequently,
\begin{align*}
\frac{\log \delta(f^{kN+\ell}x)}{kN+\ell}&=\frac{\log\left(m(Df)^N\cdot r\cdot {\rm e}^{-\varepsilon}\cdot \left(T(f^{-(\kappa(x)+\ell)}\circ f^{kN+\ell}(x))\right)^{-1}\right)}{kN+\ell}\\
&=\frac{\log\left(m(Df)^N\cdot r\cdot {\rm e}^{-\varepsilon}\right)}{kN+\ell}+\frac{\log T\left(f^{kN}\circ f^{-\kappa(x)}(x)\right)^{-1}}{kN+\ell}.
\end{align*}
Since $m(Df)^N\cdot r\cdot {\rm e}^{-\varepsilon}$ is bounded from above, we know 
$$
\lim_{k\to +\infty}\frac{\log\left(m(Df)^N\cdot r\cdot {\rm e}^{-\varepsilon}\right)}{kN+\ell}=0.
$$
According to \eqref{temt}, it follows that 
$$
\lim_{k\to +\infty}\frac{\log T\left(f^{kN}\circ f^{-\kappa(x)}(x)\right)^{-1}}{kN+\ell}=0.
$$ 
Therefore, using the arbitrariness of $-\kappa(x)\le \ell <N-\kappa(x)$ we get
$$
\lim_{n\to +\infty}\frac{1}{n}\log \delta(f^nx)=0,\quad x\in \cR.
$$
To show the convergence when $n\to -\infty$, it suffices to note that 
$$
f^{-(\kappa(x)+\ell)}\circ f^{-kN+\ell}(x)=f^{-kN}\circ f^{-\kappa(x)}(x)\in R^{\nu}, \quad \forall k\ge 1, \forall -\kappa(x)\le \ell <N-\kappa(x).
$$
Together with the convergence \eqref{temt} when $n\to -\infty$, we deduce that 
$$
\lim_{n\to -\infty}\frac{1}{|n|}\log \delta(f^nx)=0,\quad x\in \cR.
$$

For every $x\in \cR$, let
$$
W^F(x,f)=\bigcup_{n\ge 0}f^{-n}\left(W^F_{\rm loc}(f^nx,f)\right).
$$
We show first that $W^F(x,f)$ is an injectively immersed submanifold. To this end, set 
$$
\Delta_p=\{x\in \cR: \delta(x)\ge 1/p\},\quad p\in \NN.
$$
Then, we get $\cR=\bigcup_{p\in \NN}\Delta_p$. By the Poincar\'e recurrence theorem (see e.g. \cite[Theorem 1.4]{wa82}), we can assume without loss of generality that each point of $\Delta_p$, $p\in \NN$ is recurrent. That is, for every $x\in \Delta_p$, there exist infinitely many $n_k$ with $f^{n_k}(x)\in \Delta_p$ for every $k\in \NN$. For every $x\in\cR$, we choose $p\in \NN$ satisfying $x\in \Delta_p$. Fix any $m\in \NN$. As the forward iterates of any local stable manifold $W^F_{\rm loc}(f^ix,f)$ shrink in size, and by the definition of $\Delta_p$, we can choose $n$ sufficiently large so that $f^n(x)\in \Delta_p$ and
$$
f^{n-i}\left(W^F_{\rm loc}(f^i(x),f)\right)\subset W^F_{1/p}(f^{n}(x),f),\quad 0\le i \le m.
$$
This implies that 
$$
\bigcup_{0\le i \le m}f^{-i}\left(W^F_{\rm loc}(f^i(x),f)\right)\subset f^{n}\left(W^F_{1/p}(f^n(x),f)\right). 
$$
Observe that the latter is an embedded submanifold, this shows that $
W^F(x,f)=\bigcup_{n\ge 0}f^{-n}\left(W^F_{\rm loc}(f^nx,f)\right)
$ is an injectively immersed submanifold.

\medskip
To end the proof of this theorem, it suffices to show the following result.

\begin{proposition}Under the assumption of Theorem \ref{main:stable}, we have 
$$
W^F(x,f)=\left\{y\in M: \limsup_{n\to +\infty}\frac{1}{n}\log d(f^nx, f^ny)\le \chi_F^{+}(\mu,f)\right\},
$$
at $\mu$-almost every point $x$.
\end{proposition}

\begin{proof}
Let $\varepsilon_j=1/j, j\in \NN$. By the argument above we know that there exists $j_0\in \NN$ large enough so that 
\begin{equation}\label{joc}
\varepsilon_{j_0}=\frac{1}{j_0}<\min\left\{-\frac{\chi_F^-(\mu,f)}{10},\frac{\chi_E^-(\mu,f)-\chi_F^+(\mu,f)}{10}\right\},
\end{equation}
and exhibiting the following property: For every $j\ge j_0$, there exists $N_j\in \NN$, and an ergodic component $\nu_j$ of $\mu$ w.r.t. $g_j:=f^{N_j}$ such that 
$$
\nu_j\left(\bigcup_{t\ge 1}\Lambda_t^{\frac{N_j\varepsilon_j}{2}}\right)=1,
$$
where $\Lambda_t^{\frac{N_j\varepsilon_j}{2}}, t\ge 1$ are resonance blocks with respect to $(g_j,\nu_j, E\oplus F)$. Without loss of generality, we may assume that every point in each resonance bock is recurrent. 
There exists a full $\mu$-measure subset $\cR_j=\cR(\varepsilon_j)$ as a subset of $\bigcup_{i=0}^{N_j-1}f^i\big(\bigcup_{t\ge 1}\Lambda_t^{\frac{N_j\varepsilon_j}{2}}\big)$, with a positive Borel function $A_j(x)$ such that for every $x\in \cR_j$, there is a local stable manifold $W^{F,j}_{\rm loc}(x,f)$, satisfying
$$
d(f^ny,f^nz)\le A_j(x)\cdot {\rm e}^{(\chi_F^+(\mu,f)+\varepsilon_j)n}\cdot d(y,z),\quad \forall n\ge 1.
$$
when $y,z\in W^{F,j}_{\rm loc}(x,f)$. 

Now we take 
$$
R_f=\bigcap_{j\ge j_0}{\cR_j}.
$$
Hence $\mu(R_f)=1$ by construction.
For every $x\in R_f$, we consider
$$
W^{F,j}(x,f)=\bigcup_{n\ge 0}f^{-n}\left(W^{F,j}_{\rm loc}(f^nx,f)\right),\quad j\ge j_0.
$$
and
$$
W(x,f)=\left\{y\in M: \limsup_{n\to +\infty}\frac{1}{n}\log d(f^nx, f^ny)\le \chi_F^{+}(\mu,f)\right\}.
$$

We show first that $W(x,f)\subset W^{F,j}(x,f)$ for every $j\ge j_0$ at $\mu$-almost every point $x$.    For any $x\in R_f$, we have $x\in \cR_j$ for each $j\ge j_0$. So there exists $0\le i \le N_j-1$ such that $f^{-i}(x)\in R^{\nu_j}$. Thus, there exists $t\ge 1$ such that $f^{-i}(x)\in \Lambda_t^{\frac{N_j\varepsilon_j}{2}}$.
Let
$$
0<\alpha<\chi_E^-(\mu,f)-\chi_F^+(\mu,f)-\varepsilon_j,
$$
and
$$
L_m=\left\{y\in M: d(f^nx,f^ny)\cdot {\rm e}^{-(\chi_F^+(\mu,f)+\alpha)n}<{\rm e}^{-\frac{1}{2}n\alpha},\quad \forall n\ge m\right\},\quad m\in \NN.
$$
By definition, we get $W(x,f)\subset \bigcup_{m\ge 0}L_m$ and 
\[\lim_{n\to +\infty}{\rm e}^{-(\chi_F^+(\mu,f)+\alpha)n}\cdot {\rm diam}(f^n(L_m))=0,\quad \forall m\in \NN. \]
It follows that 
\begin{align*}
& \lim_{n\to +\infty}{\rm e}^{-(\chi_F^+(\nu_j,g_j)+N_j\alpha)n}\cdot {\rm diam}\left(g_j^n(f^{-i}L_m)\right)\\
=& \lim_{n\to +\infty}{\rm e}^{-(\chi_F^+(\mu,f)+\alpha)nN_j}\cdot {\rm diam}(f^{nN_j-i}(L_m))\\
=& 0.
\end{align*} 
Fix any $m\in \NN$. Since $f^{-i}x\in f^{-i}L_m\cap\Lambda_t^{\frac{N_j\varepsilon_j}{2}}$ and $\tau:=\chi_F^+(\nu_j,g_j)+N_j\alpha<\chi_E^-(\nu_j,g_j)-2\cdot \frac{N_j\varepsilon_j}{2}$, we can apply Proposition \ref{pro:forstableff} to $g_j$ and $f^{-i}L_m$, which gives some $n_k\in \NN$ such that 
$$
f^{-i}L_m\subset g_j^{-n_k}\left(W^{F,j}_{\rm loc}(g_j^{n_k}(f^{-i}x),g_j)\right).
$$
Observe from the definition in \eqref{def:staf} that $f^i\left(W^{F,j}_{\rm loc}(g_j^{n_k}(f^{-i}x),g_j)\right)=W^{F,j}_{\rm loc}(g_j^{n_k}(x),f)$. Thus
$$
L_m\subset f^i\circ g_j^{-n_k}\left(W^{F,j}_{\rm loc}(g_j^{n_k}(f^{-i}x),g_j)\right)=g_j^{-n_k}\left(W^{F,j}_{\rm loc}(g_j^{n_k}(x),f)\right)\subset W^{F,j}(x,f).
$$
Hence, we get $W(x,f)\subset W^{F,j}(x,f)$ for every $x\in R_f$ and $j\ge j_0$.

Now we show the other direction for $x\in R_f$. This will be done after showing 
\begin{equation}\label{eqsjl}
W^{F,j}(x,f)=W^{F,\ell}(x,f),\quad \forall j\neq \ell \ge j_0.
\end{equation}
Indeed, we can deduce from definition directly that 
$$
W^{F,j}(x,f)\subset \left\{y\in M: \limsup_{n\to +\infty}\frac{1}{n}\log d(f^nx, f^ny)\le \chi_F^{+}(\mu,f)+\varepsilon_j\right\}. 
$$
Following from \eqref{eqsjl} we see that for every $y\in W^{F,j}(x,f)$, we have
$$
\limsup_{n\to +\infty}\frac{1}{n}\log d(f^nx, f^ny)\le \chi_F^{+}(\mu,f)+\varepsilon_{\ell},\quad \forall \ell \ge j_0.
$$
By passing to the limit as $\ell \to +\infty$, we obtain
$$
\limsup_{n\to +\infty}\frac{1}{n}\log d(f^nx, f^ny)\le \chi_F^{+}(\mu,f),
$$
which implies $y \in W(x,f)$, and hence $W^{F,j}(x,f) \subset W(x,f)$.  Now we show \eqref{eqsjl} as follows. Fix any $x\in R_f$. Then $x\in \cR_{\ell}$ and thus there exists $0\le i \le N_{\ell}$ satisfying $f^{-i}x\in \Lambda_t^{\frac{N_{\ell}\varepsilon_{\ell}}{2}}$ for some $t\ge 1$. Let 
$$
B_m=f^{-m-i}\left(W^{F,j}_{\rm loc}(f^m(x),f)\right),\quad m\in \NN.
$$
The contracting property of $W^{F,j}_{\mathrm{loc}}(f^m(x),f)$ implies that for every $m \in \mathbb{N}$,
$$
\lim_{n\to +\infty}{\rm e}^{-\left(\chi_F^+(\mu,f)+2\varepsilon_{j}\right)n}\cdot {\rm diam}\left(f^n(B_m)\right)=0.
$$
In particular, upon replacing $n$ by $N_\ell n$, we obtain
$$
\lim_{n\to +\infty}{\rm e}^{-\left(\chi_F^+(\nu_{\ell},g_{\ell})+2\varepsilon_j\cdot N_{\ell}\right)n}\cdot {\rm diam}\left(g_\ell^n(B_m)\right)=0.
$$
By the choice of $j_0$, we know 
$
\chi_F^+(\nu_{\ell},g_{\ell})+2\varepsilon_j\cdot N_{\ell}<\chi_E^-(\nu_{\ell},g_{\ell})-\varepsilon_{\ell}\cdot N_{\ell}.
$
Therefore, by applying Proposition \ref{pro:forstableff} for $g_{\ell}$ and $B_m$ we conclude that there exists $n_k$ such that 
$$
B_m\subset g_{\ell}^{-n_k}(W^{F,j}_{\rm loc}(g_{\ell}^{n_k}(f^{-i}x),g_{\ell})).
$$
Consequently, we get for every $m\in \NN$ that
\begin{align*}
f^{-m}\left(W^{F,j}_{\rm loc}(f^m(x),f)\right)&=f^i(B_m)\subset f^i\circ g_{\ell}^{-n_k}\left(W^{F,j}_{\rm loc}(g_{\ell}^{n_k}(f^{-i}x),g_{\ell})\right)\\
&=g_{\ell}^{-n_k}\left(W^{F,j}_{\rm loc}(g_{\ell}^{n_k}(x),f)\right)\\
&\subset W^{F,{\ell}}(x,f).
\end{align*}
Since 
$$
W^{F,j}(x,f)=\bigcup_{m\ge 0}f^{-m}\left(W^{F,j}_{\rm loc}(f^m(x),f)\right),
$$
the above argument yields 
$W^{F,j}(x,f)\subset W^{F,\ell}(x,f)$. Therefore, the equality $W^{F,j}(x,f)= W^{F,\ell}(x,f)$ follows from the arbitrariness of $j$ and $\ell$. 
\end{proof}

The proof of Theorem \ref{main:stable} is now complete.

\section{Shadowing and closing lemmas}

\subsection{Dynamics on admissible manifolds}

Throughout this subsection, we assume that $g$ is a $C^1$ diffeomorphism on $M$. Let $\nu$ be a $g$-ergodic measure and $\U$ be an open neighborhood of ${\rm supp}(\nu)$, for which there exists a continuous invariant splitting 
$
T_{\U}M=E\oplus F.
$

\begin{definition}
Given $x,y\in \U$, and positive constants $\theta,r,h$, 
$D(y)$ is called an {\em $(F,\theta,r,h)$-admissible manifold near $x$}, if 
\begin{itemize}
\item $D(y)$ is a $C^1$ submanifold tangent to $\cC_{\theta}^F$.
\item $D(y)$  contains a disk of radius $r$ centered at  $y$ and $d(x,y)\le h$.
\end{itemize} 
Likewise, one can define $(E,\theta,r,h)$-admissible manifold near $x$.
\end{definition}

\begin{proposition}\label{pro:adF1}
	Let $g\in {\rm Diff}^1(M)$ preserving an ergodic measure  $\nu$. Assume that $\U$ is an open neighborhood of ${\rm supp}(\mu)$, for which there exists a continuous invariant splitting 
	$
	T_{\U}M=E\oplus F
	$
	with 
	$$
	\chi_E^-(\nu,g)>0>\chi_F^{+}(\nu,g).
	$$
	Then for every sufficiently small $\varepsilon>0$, there exist $r>r_1>0$ such that for any $t\ge 1$, one can find $h_1=h_1(t)>0$, $N_1=N_1(t)\ge 1$, and $\tau=\tau(t)\in (0,1)$ with the following property.
	
	For any $0<h\le h_1$ and any resonance block $\Lambda_t^{\varepsilon}$ w.r.t. $(g,\nu,E\oplus F)$, if $x,g^n(x)\in \Lambda^{\varepsilon}_t$ for some $n\ge N_1$, then the following holds:
	Let $D(z)$ be an $(F,\theta_1,r_1\cdot t^{-1},h)$-admissible manifold near $g^n(x)$, and let $g^n(w)$ be the transverse intersection of $D(z)$ with $W_{r_1\cdot t^{-1}}^E(g^nx,g)$. Define
	\begin{align*}
		D_n&=D(z),\\
		D_k&=g^{-1}(D_{k+1})\cap \overline{B}(g^kx,r),\quad \forall 0\le k \le n-1.
	\end{align*}
	If $D(w)\subset D_0$ is a disk of radius $r_1\cdot t^{-1}$ around $w$, then $D(w)$ is an $(F,\theta_1,r_1\cdot t^{-1}, \tau \cdot h)$-admissible manifold near $x$, and satisfies
	\begin{equation}\label{item:contraction1}
		d(g^kz_1, g^kz_2)\le t\cdot {\rm e}^{(\chi_F^+(\nu,g)+2\varepsilon)k}d(z_1,z_2),\quad \forall z_1,z_2\in D(w),\quad \forall 0\le k \le n.
	\end{equation}
\end{proposition}

Symmetrically, we can obtain

\begin{proposition}\label{pro:adE1}
Under the assumption of Proposition \ref{pro:adF1},
 for every sufficiently small $\varepsilon>0$, there exist $r>r_1>0$ such that for any $t\ge 1$, one can find $h_1=h_1(t)>0$, $N_1=N_1(t)\ge 1$, and $\tau=\tau(t)\in (0,1)$ with the following property.
 
 For any $0<h\le h_1$  and any resonance block $\Lambda_t^{\varepsilon}$ w.r.t. $(g,\nu,E\oplus F)$, if $x,g^n(x)\in \Lambda_t^{\varepsilon}$ for some $n\ge N_1$, then the following holds:
Let $D(z)$ be a $(E,\theta_1,r_1\cdot t^{-1},h)$-admissible manifold near $x$ and denote by $g^{-n}(w)$ the transverse intersection of $D(z)$ and $W_{r_1\cdot t^{-1}}^F(x,g)$. Define
\begin{align*}
D_0&=D(z),\\
D_k&=g(D_{k-1})\cap \overline{B}(g^kx,r),\quad \forall 1\le k \le n.
\end{align*}
If $D(w)\subset D_n$ is a disk of radius $r_1\cdot t^{-1}$ around $w$, then $D(w)$ is an $(E,\theta_1,r_1\cdot t^{-1}, \tau \cdot h)$-admissible manifold near $g^nx$ satisfying 
\begin{equation}\label{item:contraction}
d(g^{-k}z_1, g^{-k}z_2)\le t\cdot {\rm e}^{-(\chi_E^-(\nu,g)-2\varepsilon)k}d(z_1,z_2),\quad \forall z_1,z_2\in D(w),\quad \forall 0\le k \le n.
\end{equation}
\end{proposition}

\subsubsection{Auxiliary constants for $(g,\nu)$}\label{subsec:acf}
Let us fix $(g,\nu)$ satisfying the assumption of Proposition \ref{pro:adF1}.
For simplicity, we will rewrite $\chi_E^-=\chi_E^-(\nu,g)$ and $\chi_F^{+}=\chi_F^{+}(\nu,g)$.
Let us choose parameters $\varepsilon, \theta, r, C_0$ with the following properties:
\begin{itemize}
\item[---] $
0<\varepsilon<\min\left\{\frac{1}{10}\chi_E^-, -\frac{1}{10}\chi_F^+,\frac{1}{10}\left(\chi_E^--\chi_F^+\right)\right\}.
$
\item[---] Choose $r_0$ such that $B({\rm supp}(\nu),r_0)\subset \U$.
\item[---] $r=r(\varepsilon)<r_0$ satisfies Lemma\ref{lem:averf}, Propositions \ref{pro:averf} ~\ref{pro:avere} and Theorem \ref{thm:un-stableblock} for the above $\varepsilon$. 
\item[---] For any submanifold $W$ tangent to $\cC_{\theta}^E$ or $\cC_{\theta}^F$, within ${\rm diam}(W)\le 2r$, we have 
\begin{equation}\label{disch}
{\rm e}^{-\varepsilon/4}d(z_1,z_2)\le d_W(z_1,z_2) \le {\rm e}^{\varepsilon/4}d(z_1,z_2),\quad \forall z_1,z_2\in W.
\end{equation}
\item[---] If $D^E$ and $D^F$ are submanifolds tangent to $\mathcal{C}_\theta^E$ and $\mathcal{C}_\theta^F$, respectively, and each has diameter no less than $2r$, then $D^E$ and $D^F$ intersect in at most one point. When $D^E\cap D^F=\{z\}$, then 
$$
\max\{d(x,z), d(y,z)\}\le C_0\cdot d(x,y),
$$
for every $x\in D^E$ and $y\in D^F$.
\item[---] Denote $r_1=r\cdot {\rm e}^{-\varepsilon}$.
\end{itemize}

For every $t\ge 1$, let us fix $\theta_1=\theta_1(t)\le \theta$ to satisfy Propositions \ref{pro:averf} ~\ref{pro:avere}. In particular, we have
\begin{equation}\label{esteps}
\frac{\|Dg^{-1}|_{\widetilde{E}(x_1)}\|}{\|Dg^{-1}|_{E(x_2)}\|}, \frac{m\left(Dg^{-1}|_{\widetilde{F}(x_1)}\right)}{m\left(Dg^{-1}|_{F(x_2)}\right)}\in \left({\rm e}^{-\frac{\varepsilon}{2}}, {\rm e}^{\frac{\varepsilon}{2}}\right),
\end{equation}
whenever $x_1,x_2\in B({\rm supp}(\nu),r_0)\subset \U$ with $d(x_1,x_2)<r$ 
and $\widetilde{E}\subset \cC^E_{t^2\theta_1(t)}$ and $\widetilde{F}\subset \cC^F_{t^2\theta_1(t)}$.

By the continuity of subbundles, we have the following observation.

\begin{lemma}\label{interuniquet}
For every $t\ge 1$, there exists $h_0(t)>0$ such that for every $h\le h_0(t)$, for every $x,y,z\in B({\rm supp}(\nu),r_0)$, if $D^E(y), D^F(z)$ are $(E,\theta_1,r_1\cdot t^{-1},h)$ and $(F,\theta_1,r_1\cdot t^{-1},h)$ admissible manifolds near $x$, respectively, then $D^E(y)$ intersects $D^F(z)$ transversely at a unique point.
\end{lemma}

%

\subsubsection{Proof of Proposition \ref{pro:adF1}}

Fix constants as in Subsection \ref{subsec:acf}.
Let $h_0(t)$ be the constant given by Lemma \ref{interuniquet}. Then, take $h_1=h_1(t)<h_0(t)$ such that 
$$
t\cdot C_0\cdot h_1+r_1<r,\quad {\rm e}^{\varepsilon/2}\cdot C_0\cdot h_1<\frac{1}{2}\cdot r_1\cdot t^{-1}.
$$
Choose integer $N_1\ge 1$ such that 
\begin{equation}\label{n11}
r_1\cdot {\rm e}^{(\chi_F^++3\varepsilon/2)N_1}<\frac{1}{2}r_1\cdot t^{-1},
\end{equation}
 and
\begin{equation}\label{n12}
\tau:=C_0\cdot t\cdot {\rm e}^{- \min\big\{\chi_E^-(\nu,g)-2\varepsilon, -(\chi_F^+(\nu,g)+2\varepsilon)\big\}N_1}<1.
\end{equation}
For $0<h\le h_1$ and $x, g^n(x)\in \Lambda_t^{\varepsilon}$, let $D(z)$ be an $(F, \theta_1,r_1\cdot t^{-1},h)$-admissible manifold near $g^n(x)$. Since $g^n(x)\in \Lambda_t^{\varepsilon}$, it admits a local unstable manifold of radius $r\cdot t^{-1}$ guaranteed by Theorem \ref{thm:un-stableblock}. In particular,  
$$
W^E(g^nx,g)\supset W^E_{r\cdot t^{-1}}(g^nx,g)\supset W^E_{r_1\cdot t^{-1}}(g^nx,g).
$$
Let $g^n(w)$ be the unique intersection of $D(z)$ and $W^E_{r_1\cdot t^{-1}}(g^nx,g)$, guaranteed by Lemma \ref{interuniquet}. From the choice of $h_1$, one knows that $D(z)\subset B(g^nx,r)$. Let $D_n=D(z)$, and define
$$
D_k=g^{-1}(D_{k+1})\cap \overline{B}(g^kx,r),\quad \forall 0\le k \le n-1.
$$
Since $g^n(w)\in W_{r_1\cdot t^{-1}}^E(g^nx,g)$, we know for every $0\le k \le n$ that 
\begin{align}\label{35}
d(g^kw, g^kx)&\le t\cdot {\rm e}^{-(\chi_E^--2\varepsilon)(n-k)}\cdot d(g^nw,g^nx)\notag\\
&\le t\cdot C_0\cdot d(z,g^nx)\notag\\
&<t\cdot C_0\cdot h\notag\\
&\le t\cdot C_0\cdot h_1.
\end{align}
With the choice of $h_1$, the preceding argument  implies that $g^kw\in B(g^kx,r)$ for every $0\le k\le n$.
Now we show the following: 
\begin{claim}
$D_0$ has radius no less than $r_1\cdot t^{-1}$ around $w$.
\end{claim}

\begin{proof}
Assume by contradiction that there is $a\in \partial D_0$ such that $d_{D_0}(w,a)<r_1\cdot t^{-1}$. Then 
$$
g^k(a)\in D_k\subset \overline{B}(g^k(x),r),\quad \forall 0\le k \le n.
$$
By the construction of $D_k$, we know $g^k(D_0)\subset D_k\subset \overline{B}(g^k(x),r)$ for every $0\le k \le n$. Let us take $\xi_i\in g^i(D_0)$ such that 
$$
d_{g^{i-1}(D_0)}(g^{i-1}w, g^{i-1}a)\ge m(Dg^{-1}|_{T_{\xi_i}(g^iD_0)})\cdot d_{g^i(D_0)}(g^iw, g^ia),\quad 1\le i \le n.
$$
By Lemma \ref{lem:averf}, we conclude that $g^i(D_0)$ is tangent to $\cC^F_{t^2\cdot \theta_1}$ for all $1\le i \le n$, whence
$T_{\xi_i}(g^iD_0)\in \cC^F_{t^2\cdot \theta_1}(\xi_i)$.
Observe also that $d(\xi_i, g^ix)\le r$ as $g^i(D_0)\subset \overline{B}(g^ix,r)$. As a consequence, 
\begin{align*}
d_{D_0}(w,a)&\ge \left(\prod_{i=1}^{k}m(Dg^{-1}|_{T_{\xi_i}(g^iD_0)})\right)\cdot d_{g^kD_0}(g^kw, g^ka)\\
&\ge \left(\prod_{i=1}^{k}{\rm e}^{-\varepsilon/2}\cdot m(Dg^{-1}|_{F(g^ix)})\right)\cdot d_{g^kD_0}(g^kw, g^ka) \quad (\textrm{by}~\eqref{esteps})\\
&\ge t^{-1}\cdot {\rm e}^{-(\chi_F^++3\varepsilon/2)k}\cdot d_{g^kD_0}(g^kw,g^ka). \quad (\textrm{by Corollary}~\ref{cor:lambda-bl})
\end{align*}
Together with the assumption $d_{D_0}(w,a)<r_1\cdot t^{-1}$, this yields
\begin{equation}\label{eq:wadis}
d_{g^kD_0}(g^kw,g^ka)<r_1\cdot {\rm e}^{(\chi_F^++3\varepsilon/2)k}.
\end{equation}
In particular, we get
$$
d(g^ka,g^kx)\le d(g^ka,g^kw)+d(g^kw,g^kx)<r_1+t\cdot C_0\cdot h_1<r
$$
for every $0\le k \le n$.
Thus, from the facts $a\in \partial D_0$, $a\in B(x,r)$ and $D_0=g^{-1}(D_1)\cap \overline{B}(x,r)$ we conclude that $g(a)\in \partial D$, otherwise we have $a\in {\rm int}(D_0)$, which is a contradiction. 
We obtain inductively that $g^k(a)\in \partial D_k$ for every $1\le k \le n$. In particular, we see from $g^n(a)\in \partial D_n$ that $d_{D(z)}(z_1,g^na)=r_1\cdot t^{-1}$. Moreover, 
\begin{equation}\label{ezw1}
d_{D(z)}(z,g^nw)\le {\rm e}^{\varepsilon/4}\cdot d(z,g^nw)\le {\rm e}^{\varepsilon/4}\cdot C_0\cdot d(z,g^nx)<{\rm e}^{\varepsilon /4}\cdot C_0\cdot h_1.
\end{equation}
In view of \eqref{n11}, for any $n\ge N_1$, we have from \eqref{eq:wadis} that
\begin{equation}\label{ezw2}
d_{g^nD_0}(g^nw, g^na)\le r_1\cdot {\rm e}^{(\chi_F^++3\varepsilon/2)n}<\frac{1}{2}r_1\cdot t^{-1}.
\end{equation}
Hence, by combining \eqref{ezw1}, \eqref{ezw2} and the choice of $h_1$, one concludes that
\begin{align*}
r_1\cdot t^{-1}=d_{D(z)}(z,g^na)&\le d_{D(z)}(z,g^nw)+d_{D(z)}(g^nw,g^na)\\
&\le {\rm e}^{\varepsilon/4}\cdot C_0\cdot h_1+\frac{1}{2}r_1\cdot t^{-1}\\
&< r_1\cdot t^{-1},
\end{align*}
which is a contradiction. Thus we complete the proof of the claim.
\end{proof}
Denote by $D(w)$ the disk of radius $r_1 \cdot t^{-1}$ centered at $w$ inside $D_0$. Then $D(w)$ is tangent to $\mathcal{C}_{\theta_1}^F$ and satisfies $g^n(D(w)) \subset D(z)$. Since $g^n w \in W^E_{\mathrm{loc}}(g^n x, g)$,  by \eqref{35} and \eqref{n12}, we have
\[
d(w, x) \le t \cdot e^{-(\chi_E^- - 2\varepsilon)n} \cdot d(g^n w, g^n x) \le t \cdot e^{-(\chi_E^- - 2\varepsilon)n} \cdot C_0 \cdot h \le  \tau \cdot h.
\]
  Thus, $D(w)$ is an $(F,\theta_1,r_1\cdot t^{-1}, \tau \cdot h)$-admissible manifold near $x$.
It remains to show the last contraction property of $D(w)$. Indeed, similar to the estimation on $d_{g^kD_0}(g^kw,g^ka)$ in the proof of the claim, for every $0\le k \le n$, it holds that
$$
d_{g^kD(w)}(g^kz_1, g^kz_2)\le t\cdot {\rm e}^{(\chi_F^++3\varepsilon/2)k}d_{D(w)}(z_1,z_2),\quad \forall z_1,z_2\in D(w).
$$ 
This gives \eqref{item:contraction1}, consulting \eqref{disch}.

\subsection{Shadowing and closing on resonance blocks}


\begin{theorem}\label{th:shaodowg}
Let $g\in {\rm Diff}^1(M)$  preserving an ergodic measure  $\nu$. Assume that $\U$ is an open neighborhood of ${\rm supp}(\nu)$, for which there exists a continuous invariant splitting 
$
T_{\U}M=E\oplus F
$
with 
$$
\chi_E^-(\nu,g)>0>\chi_F^{+}(\nu,g).
$$

Then for any $\varepsilon>0$ sufficiently small, and any $t\ge 1$, there exist $\beta_0=\beta_0(t)>0$, $C_1=C_1(t)>0$, $N_2=N_2(t)\ge 1$ such that for every $0<\beta<\beta_0$, if $\{(x_k,n_k)\}_{k\in \ZZ}$ is a $\beta$-pseudo-orbit in a resonance block $\Lambda_t^{\varepsilon}$ w.r.t. $(g,\nu,E\oplus F)$  satisfying $n_k\ge N_2$ for all $k$, then there is a unique $z\in M$ whose orbit  $(C_1\beta,\lambda)$- shadows $\{(x_k,n_k)\}_{k\in\ZZ}$ with rate $\lambda=\min\{\chi_E^-(\nu,g)-2\varepsilon, -(\chi_F^+(\nu,g)+2\varepsilon)\}$. 

Moreover, if $\{(x_k,n_k)\}_{k\in \ZZ}$ is periodic, then the shadowing point $z$ can be chosen to be a hyperbolic periodic  point.
\end{theorem}

\begin{proof}
Fix $\varepsilon>0$ sufficiently small such that 
$$
0<\varepsilon<\min\left\{\frac{1}{10}\chi_E^-(\nu,g), -\frac{1}{10}\chi_F^+(\nu,g),\frac{1}{10}\left(\chi_E^-(\nu,g)-\chi_F^+(\nu,g)\right)\right\}.
$$
Thus, $\lambda=\min\left\{\chi_E^-(\nu,g)-2\varepsilon, -(\chi_F^+(\nu,g)+2\varepsilon)\right\}>0$.
For any $t\ge 1$ we fix size $r$ as in Subsection \ref{subsec:acf}. Let $h_1=h_1(t), N_1=N_1(t),\tau=\tau(t)\in (0,1)$ be the constants given by Propositions \ref{pro:adF1} and \ref{pro:adE1}.
Let
\begin{align*}
C_1=\frac{1+C_0+3C_0^2}{1-\tau}\cdot t, ~\text{and}~~
\beta_0&=\frac{(1-\tau)}{C_1}\cdot h_1.
\end{align*}
Choose $N_2\ge N_1$ such that 
$$
\alpha:=t\cdot {\rm e}^{-\lambda N_2}\in (0,1), ~\text{and}~~t^{-1}\ge {\rm e}^{-\frac{1}{2}\varepsilon N_2}.
$$
For any $0<\beta<\beta_0$, let $\{(x_k,n_k)\}_{k\in \ZZ}$ be a $\beta$-pseudo-orbit in $\Lambda_t^{\varepsilon}$ satisfying $n_k\ge N_1$ for all $k$. Now we fix any $m\in \NN$. Let $D^F(x_m)$ be a disk inside $W^F_{loc}(x_m,g)$ of radius $r_1\cdot t^{-1}$ centered at $x_m$. By definition, $D^F(x_m)$ is an $(F,\theta_1,r_1\cdot t^{-1},\beta)$-admissible manifold near $g^{n_{m-1}}(x_{m-1})$. By Proposition \ref{pro:adF1}, there is an $(F,\theta_1,r_1\cdot t^{-1},\beta\cdot \tau)$-admissible manifold $D^F(w_m^{m-1})$ near $x_{m-1}$ such that 

$$
g^{n_{m-1}}(D^F(w_m^{m-1}))\subset D^F(x_m),
$$ 
$$
d(g^iz_1,g^iz_2)\le t\cdot {\rm e}^{(\chi_F^+(\nu,g)+2\varepsilon)i}\cdot d(z_1,z_2),
$$
for every $z_1,z_2\in D^F(w_m^{m-1})$ and $0\le i\le n_{m-1}$.
Note that $\beta(1+\tau)<(1-\tau)^{-1}\cdot \beta<h_1$. We can use Proposition \ref{pro:adF1} inductively to show that for every $-m\le k < m$, there exists an $\left(F,\theta_1,r_1\cdot t^{-1},\beta(\tau+\cdots+\tau^{m-k})\right)$-admissible manifold $D^F(w_m^k)$ near $x_k$, satisfying 
\begin{itemize}
\item[---] $g^{n_{k}}(D^F(w_m^{k}))\subset D^F(w_m^{k+1})$,
\item[---] $d(g^iz_1,g^iz_2)\le t\cdot {\rm e}^{(\chi_F^+(\nu,g)+2\varepsilon)i}\cdot d(z_1,z_2),\quad \forall z_1,z_2\in D^F(w_m^{k}),\quad \forall 0\le i\le n_{k}$,
\end{itemize}
where $w_m^m=x_m$. 

Similarly, if we let $D^E(x_{-m})$ be a disk inside $W^E_{loc}(x_{-m},g)$ of radius $r_1\cdot t^{-1}$ centered at $x_{-m}$. Then $D^E(x_{-m})$ is a $(E,\theta_1,r_1\cdot t^{-1},0)$-admissible manifold near $x_{-m}$. By Proposition \ref{pro:adE1}, for every $-m< k \le m$, one can find an $\left(E,\theta_1,r_1\cdot t^{-1},\beta(1+\tau+\cdots+\tau^{m+k-1})\right)$-admissible manifold $D^E(w_{-m}^{k})$ near $x_{k}$ satisfying 
\begin{itemize}
\item[---]$g^{-n_{k-1}}(D^E(w_{-m}^{k}))\subset D^E(w_{-m}^{k-1})$,
\item[---] $d(g^{-i}z_1,g^{-i}z_2)\le t\cdot {\rm e}^{-(\chi_E^-(\nu,g)-2\varepsilon)i}\cdot d(z_1,z_2),\quad \forall z_1,z_2\in D^E(w_{-m}^{k}),\quad \forall 0\le i\le n_{k-1}$,
\end{itemize}
where $w_{-m}^{-m}=x_{-m}$.

Let us consider 
$$
z_m^k=D^E(w_{-m}^{k})\cap D^F(w_{m}^{k}),\quad \forall -m\le k \le m.
$$

\begin{figure}[h]
	\centering
	\includegraphics[width=0.85\textwidth]{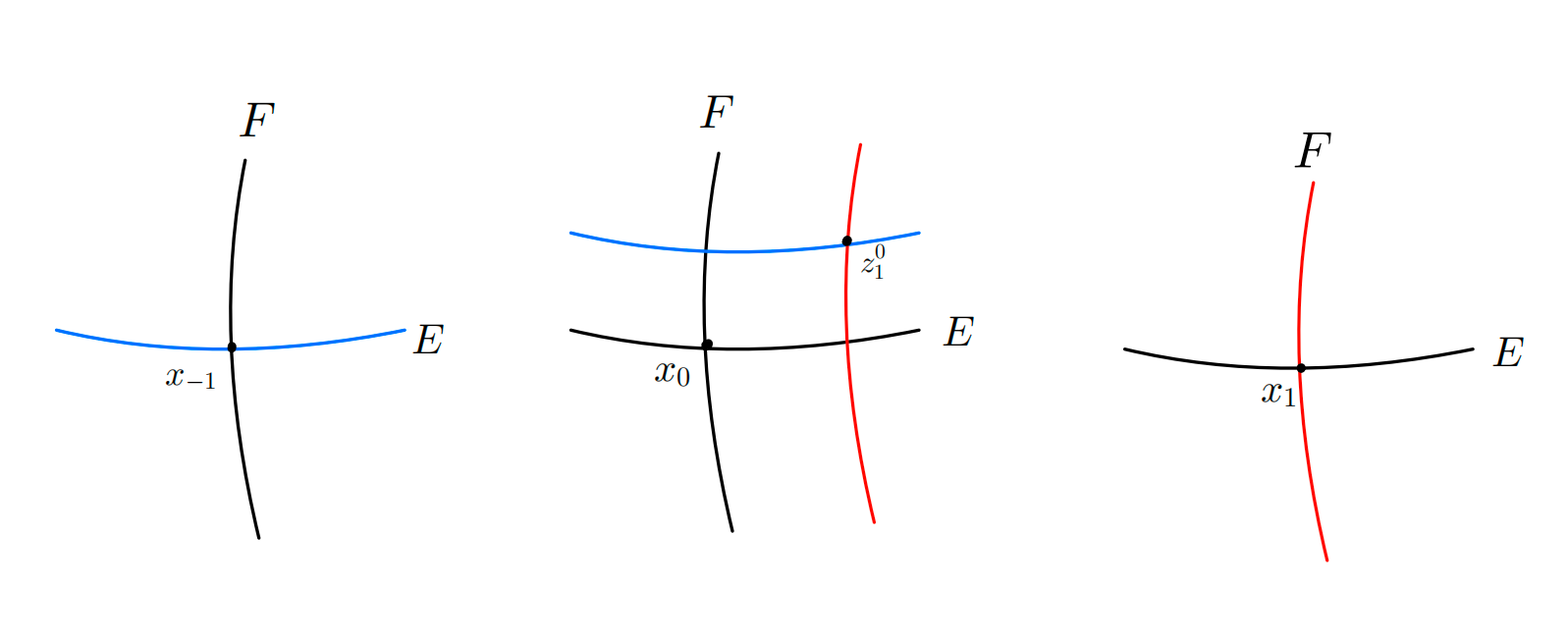}
	\caption{Position of $z_1^0$}
\end{figure} 
We split the proof into several steps.
\paragraph{Step 1.}
We prove that 
$$
g^{n_k}(z_m^k)=z_m^{k+1},\quad  \forall -m\le k < m.
$$

From $z_m^k\in D^F(w_m^k)$ and the construction of $D^F(w_m^k)$ in Proposition \ref{pro:adF1}, it follows that $g^{n_k}(z_m^k)\in D^F(w_m^{k+1})$ and 
\begin{equation}\label{stepone-11}
g^i(z_m^k)\in g^i(D^F(w_m^k))\subset \overline{B}(g^i(x_k),r).
\end{equation}
Thus, it suffices to show $g^{n_k}(z_m^k)\in D^E(w_{-m}^{k+1})$.
For any fixed $m,k$, take
\begin{align*}
D_0&=D^E(w_{-m}^k),\\
D_i&=g(D_{i-1})\cap \overline{B}(g^ix_k,r),\quad 1\le i \le n_k.
\end{align*}
Then $D^E(w_{-m}^{k+1})$ is a disk inside $D_{n_k}$ of radius $r_1\cdot t^{-1}$ centered at $w_{-m}^{k+1}$. It follows from \eqref{stepone-11} that 
$$
g^i(z_m^k)\in D_i,\quad \forall 0\le i \le n_k,
$$
which implies $g^{n_k}(z_m^k)\in D_n\cap D^F(w_m^{k+1})$.
Observe that both $D_{n_k}$ and $D^F(w_m^{k+1})$ have diameters no larger than $2r$, and are tangent to 
$\cC^E_{\theta_1}$ and $\cC^F_{\theta_1}$, respectively. So by the choice of $r$ one knows that $D_n$ intersects $D^F(w_m^{k+1})$ in at most one point. Hence,
$$
g^{n_k}(z_m^k)=D_{n_k}\cap D^F(w_m^{k+1}).
$$
Note that $D^E(w_{-m}^{k+1})\subset D_n$ and $D^E(w_{-m}^{k+1})\cap D^F(w_m^{k+1})=z_m^{k+1}$. As a result, we obtain 
$
g^{n_k}(z_m^k)=z_m^{k+1}.
$

\paragraph{Step 2.}
We prove   for every $0\le j\le n_k$ and $-m\le k \le m-1$, it holds that 
$$
d(g^j(z_m^k),g^j(x_k))\le C_1\cdot \beta\cdot {\rm e}^{-\lambda \min\{j,n_k-j\}}.
$$
 By the construction of $D^F(w_m^k)$, we have
$w_m^k=W^E_{r_1\cdot t^{-1}}(x_k,g)\cap D^F(w_m^k).$ Denote
\[\quad b_{-m}^k=W^F_{r_1\cdot t^{-1}}(x_k,g)\cap D^E(w_{-m}^k).\]
Then, from the construction of $D^E(w_{-m}^k)$, it follows that $g^{n_k}(b_{-m}^k)=w_{-m}^{k+1}$.
\begin{figure}[h]
	\centering
	\includegraphics[width=0.60\textwidth]{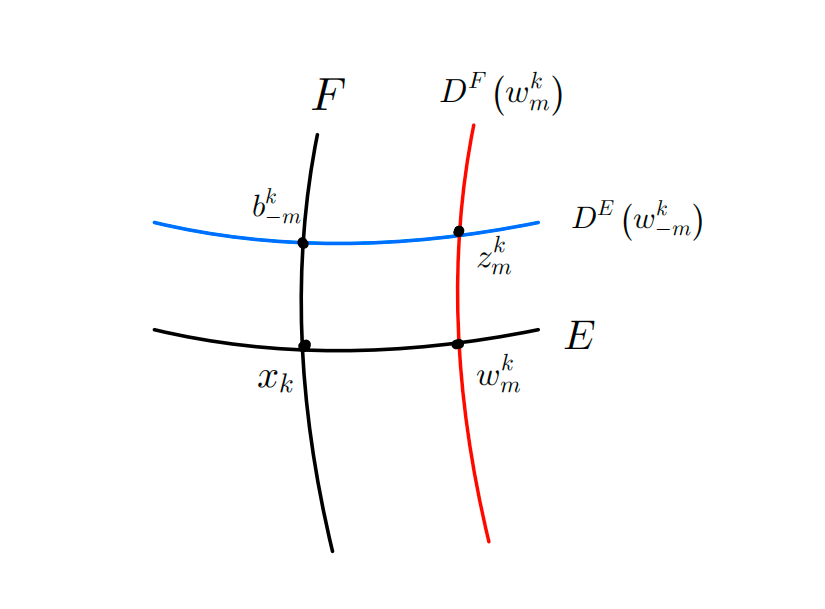}
	\caption{Positions of $w_m^k$ and $b_{-m}^k$}
\end{figure} 
Moreover,
$$
d(x_k,w_m^k)\le  \beta\left(\tau+\tau^2+\cdots+\tau^{m-k}\right)<\frac{1}{1-\tau}\cdot \beta<\frac{C_0}{1-\tau}\cdot \beta,
$$
$$
d(x_k,b_{-m}^k)\le C_0\cdot d(x_k,w_{-m}^k)\le C_0\cdot \beta\left(1+\tau+\cdots+\tau^{m+k-1}\right)<\frac{C_0}{1-\tau}\cdot \beta.
$$
Therefore, 
$$
d(w_{m}^k,z_m^k)\le C_0\cdot d(w_m^k,b_{-m}^k)\le C_0\cdot \left(d(w_m^k,x_k)+d(x_k,b_{-m}^k)\right)<\frac{2C_0^2}{1-\tau}\cdot \beta.
$$
Consequently,
$$
d(x_k,z_m^k)\le d(x_k,w_m^k)+d(w_m^k,z_m^k)\le \frac{3C_0^2}{1-\tau}\cdot \beta.
$$
Since $b_{-m}^k\in W^F_{r_1\cdot t^{-1}}(x_k,g)$ and $x_k\in\Lambda_t^\varepsilon$, the choice of $\lambda$ implies
 
\begin{equation}\label{41}
	d\left(g^i(x_k), g^i(b_{-m}^k)\right)\le t\cdot {\rm e}^{(\chi_F^+(\nu,g)+2\varepsilon)i}\cdot d(x_k,b_{-m}^k)\le t\cdot {\rm e}^{-\lambda\cdot i}\cdot \frac{C_0}{1-\tau}\cdot \beta.
\end{equation}
Similarly,  as    $g^{n_k}(b_{-m}^k)=w_{-m}^{k+1}\in   D^E(w_{-m}^{k+1})$ and $ g^{n_k}(z_m^k)=z_m^{k+1}\in   D^E(w_{-m}^{k+1})$, 
\begin{align*}
d(g^i(b_{-m}^k),g^i(z_m^k))\le & t\cdot {\rm e}^{-(\chi_E^-(\nu,g)-2\varepsilon)(n_k-i)}\cdot d(g^{n_k}(b_{-m}^k), g^{n_k}(z_m^k))\\
\le& t\cdot {\rm e}^{-\lambda(n_k-i)}\cdot d(w_{-m}^{k+1}, z_m^{k+1})\\
\le& t\cdot {\rm e}^{-\lambda(n_k-i)}\cdot\big(d(w_{-m}^{k+1},x_{k+1})+d(x_{k+1},z_m^{k+1})\big)\\
\le& t\cdot {\rm e}^{-\lambda(n_k-i)}\cdot \big(  \frac{\beta}{1-\tau} +\frac{3C_0^2}{1-\tau}\cdot \beta\big)\\
=&\frac{1+3C_0^2}{1-\tau}\cdot \beta\cdot t \cdot {\rm e}^{-\lambda(n_k-i)}.
\end{align*}
Combining above two estimates, together with the choice of $C_1$, we get
\begin{align*}
d(g^i(x_k),g^i(z_m^k))&\le  d(g^i(x_k),g^i(b_{-m}^k))+d(g^i(b_{-m}^k),g^i(z_m^k))\\
&\le t\cdot {\rm e}^{-\lambda\cdot i}\cdot \frac{C_0}{1-\tau}\cdot \beta+\frac{1+3C_0^2}{1-\tau}\cdot \beta\cdot t \cdot {\rm e}^{-\lambda(n_k-i)}\\
&\le C_1\cdot \beta\cdot {\rm e}^{-\lambda \min\{i,n_k-i\}}.
\end{align*}

\paragraph{Step 3.}
  we show that ${z_m^0}{m\ge 1}$ converges.

Let
$$
y_m^k=D^E(w_{-m}^k)\cap D^F(w_{m+1}^k),\quad \forall -m\le k\le m.
$$

\begin{figure}[h]
	\centering
	\includegraphics[width=0.55\textwidth]{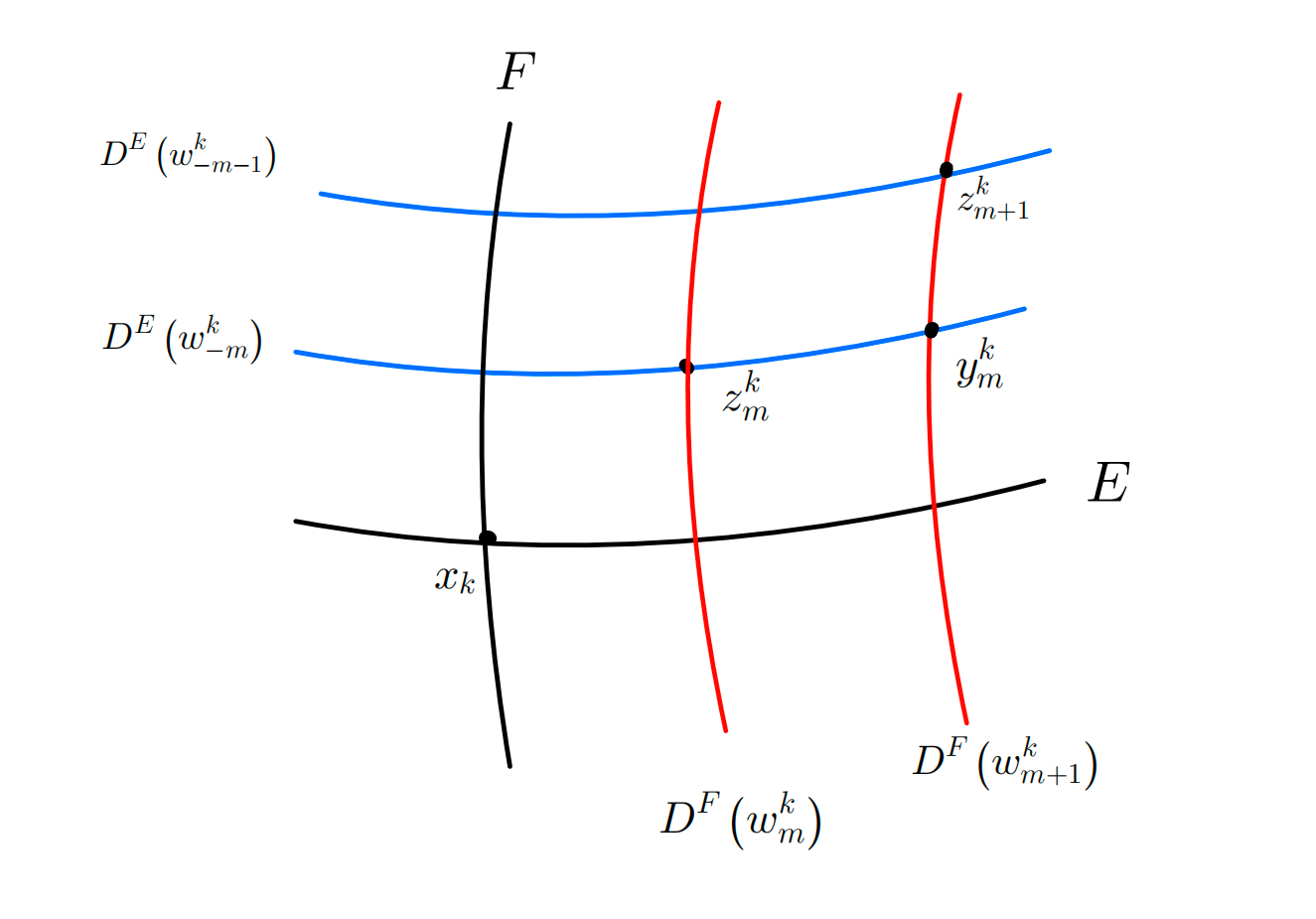}
	\caption{Position of $y_m^k$}
\end{figure} 

Similar to \textbf{Step 1}, we can prove $g^{n_k}(y_m^k)=y_m^{k+1}$. So, with the fact $z_m^{k+1},y_m^{k+1}\in D^E(w_{-m}^{k+1})$ we get
\begin{align*}
d(z_m^k,y_m^k)&=d\left(g^{-n_k}(z_m^{k+1}),g^{-n_k}(y_m^{k+1})\right)\\
&\le t\cdot {\rm e}^{-(\chi_E^-(\nu,g)-2\varepsilon)n_k}\cdot d(z_m^{k+1}, y_m^{k+1})\\
&\le \alpha \cdot d(z_m^{k+1}.y_m^{k+1}),
\end{align*}
recall that $\alpha=t\cdot {\rm e}^{-\lambda\cdot N_2}\in (0,1)$.
With the same manner, we can show 
$$
d(y_m^k,z_{m+1}^k)=d(g^{n_{k-1}}(y_m^{k-1}),g^{n_{k-1}}(z_{m+1}^{k-1}))\le \alpha\cdot d(y_m^{k-1},z_{m+1}^{k-1}).
$$
Combining these two estimates, we obtain 
\begin{align*}
d(z_m^0,z_{m+1}^0)&\le d(z_m^0,y_m^0)+d(y_m^0,z_{m+1}^0)\\
& \le  \alpha^m\cdot d(z_m^m,y_m^m)+\alpha^m\cdot d(y_m^{-m},z_{m+1}^{-m})\\
& \le  \alpha^m\cdot (2r_1\cdot t^{-1})\cdot 2.
\end{align*}
This suggests that $\{z_m^0\}_{m\ge 1}$ is a Cauchy sequence, thus converges to some point $z$ as $m\to +\infty$.
It follows from the results of \textbf{Step 1} and \textbf{Step 2} that for every $-m\le k\le m-1$ and $0\le j\le n_k$, it holds that 
$$
d\left(g^{s_k+j}(z_m^0),g^j(x_k)\right)=d\left(g^j(z_m^k),g^j(x_k)\right)\le C_1\cdot \beta\cdot {\rm e}^{-\lambda \min\{j,n_k-j\} },
$$
where
$$
\displaystyle
s_k=
\left\{
\begin{array}{ll}
  \sum_{i=0}^{k-1}n_i, &\quad \textrm{if}~~k>0, \\[2ex]
 \quad 0, &\quad \textrm{if}~~k=0,\\[2ex]
  -\sum_{i=k}^{-1}n_i, & \quad \textrm{if}~~k<0. 
\end{array}
\right.
$$
By letting $m\to \infty$, we then get 
$$
d(g^{s_k+j}(z),g^j(x_k))\le C_1\cdot \beta\cdot {\rm e}^{-\lambda \min\{j,n_k-j\}},\quad  \forall k\in \ZZ, \quad \forall 0\le j \le n_k.
$$
Now we show the uniqueness. Assume that both $z$ and $\tilde{z}$ satisfy the above inequality. Let $D^F(g^{s_m}(z))$ be an $\left(F,\theta_1,r_1\cdot t^{-1}, 0\right)$ admissible manifold near $g^{s_m}(z)$. Then, by
$$
d(g^{s_m}(z),g^{n_{m-1}}(x_{m-1}))\le C_1\beta<h_1,
$$
and using Proposition \ref{pro:adF1} inductively we see that there exists an $\big(F,\theta_1,r_1\cdot t^{-1},\beta(\tau+\cdots +\tau^{m-k-1}+\tau^{m-k}\cdot C_1)\big)$-admissible manifold $D^F(\xi_m^k)$ near $x_k$. Since $d(g^{s_k}(z),x_k)\le C_1\beta$ for every $k\in \ZZ$, by the construction of the admissible manifold we know that 
$$
g^{s_k}(z)\in D^F(\xi_m^k),\quad \forall k\le m.
$$ 
In the same manner, there exists a $\big(E,\theta_1,r_1\cdot t^{-1},\beta(1+\cdots+ \tau^{m-k-1}+\tau^{m-k}\cdot C_1)\big)$-admissible manifold $D^E(\xi_{-m}^k)$ near $x_k$, which satisfies 
$$
g^{s_k}(\tilde{z})\in D^E(\xi_{-m}^k),\quad \forall k\ge -m.
$$
Let $y=D^E(\xi_{-m}^{0})\cap D^F(\xi_{m}^0)$. Similar to the proof in \textbf{Step 1}, we deduce that 
$$
g^{s_k}(y)=D^E(\xi^k_{-m})\cap D^F(\xi^k_{m}),\quad \forall -m\le k \le m.
$$ 
Therefore, 
$$
d(g^{s_{k+1}}(y),g^{s_{k+1}}(z))\le t {\rm e}^{(\chi_F^+(\nu,g)+2\varepsilon)n_k}d(g^{s_k}(y),g^{s_k}(z))\\
\le \alpha\cdot d(g^{s_k}(y),g^{s_k}(z)).
$$
$$
d(g^{s_{k-1}}(y),g^{s_{k-1}}(\tilde{z}))\le t {\rm e}^{-(\chi_E^-(\nu,g)-2\varepsilon)n_k}d(g^{s_k}(y),g^{s_k}(\tilde{z}))\\
\le \alpha\cdot d(g^{s_k}(y),g^{s_k}(\tilde{z})).
$$
Inductively, we have 
$$
d(y,z)\le \alpha^m d(g^{s_{-m}}(y),g^{s_{-m}}(z))\le \alpha^m r_1\cdot t^{-1}.
$$
$$
d(y,\tilde{z})\le \alpha^m d(g^{s_m}(y),g^{s_m}(\tilde{z}))\le \alpha^m r_1\cdot t^{-1}.
$$
As a result, we obtain
$
d(z,\tilde{z})\le d(z,y)+d(y,\tilde{z})\le 2\alpha^m r_1\cdot t^{-1}.
$
Letting $m\to +\infty$, we get $z=\tilde{z}$ which gives the uniqueness.

\paragraph{Step 4.}
Now we show the result of the case that $\{(x_k,n_k)\}$ is periodic.

In this case, we have $x_k=x_0, n_k=n_0$ for every $k\in \ZZ$.
Thus, $x_0,g^{n_0}(x_0)\in \Lambda_t^{\varepsilon}$ with $d(x_0,g^{n_0}(x_0))<\beta$, and $n_0\ge N_2$. Put
$$
u_m^k=D^E(w_{-m}^{k+1})\cap D^F(w_m^k).
$$

\begin{figure}[h]
	\centering
	\includegraphics[width=0.55\textwidth]{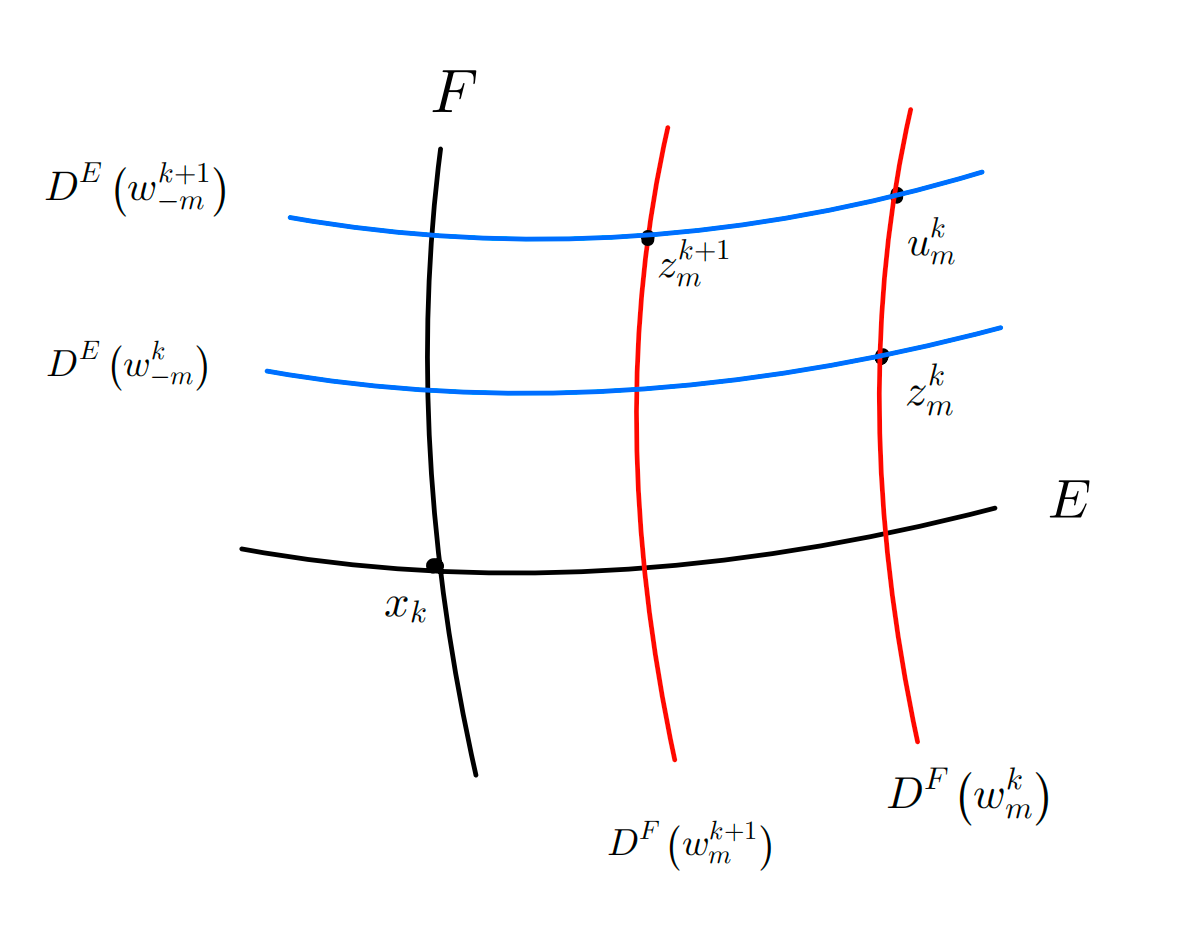}
	\caption{Position of $u_m^k$}
\end{figure}

Using the argument in \textbf{Step 3}, we can show
$$
d(z_m^{k+1},u_m^k)=d(g^{-n_{k+1}}(z_m^{k+2}), g^{-{n_{k+1}}}(u_m^{k+1}))\le \alpha\cdot d(z_m^{k+2},u_m^{k+1}),
$$
$$
d(u_m^{k},z_m^k)=d(g^{n_{k-1}}(u_m^{k-1}), g^{{n_{k-1}}}(z_m^{k-1}))\le \alpha\cdot d(u_m^{k-1},z_m^{k-1}).
$$
Consequently,
\begin{align*}
d(z_m^1,z_m^0)&\le d(z_m^1,u_m^0)+d(u_m^0,z_m^0)\\
&\le \alpha^{m-1}\cdot d(z_m^m, u_m^{m-1})+\alpha^m\cdot d(u_m^{-m},z_m^{-m})\\
&\le 2\cdot \alpha^{m-1}\cdot (2\cdot r_1\cdot t^{-1}).
\end{align*}
As $z_m^0\to z$, we then see $z_m^1\to z$ as $m\to \infty$.
Note also that $g^{n_0}(z_m^0)=z_m^1$, so $g^{n_0}(z)=z$. Now we show $z$ is a hyperbolic periodic point. Since
$$
d(g^iz,g^ix_0)\le C_1\cdot\beta<(1-\tau)h_1<r,\quad \forall 0\le i \le n_0-1,
$$
by Corollary \ref{cor:lambda-bl},  the choice of constants in \eqref{esteps}, and the choice of $N_2$, we obtain
\begin{align*}
m\left(Dg^{n_0}|_{E(z)}\right)\ge \prod_{i=0}^{n_0-1}m(Dg|_{E(g^iz)})
&\ge \prod_{i=0}^{n_0-1}\left({\rm e}^{-\varepsilon/2}m(Dg|_{E(g^ix_0)})\right)\\
&\ge t^{-1}{\rm e}^{(\chi_E^-(\nu,g)-3/2\varepsilon)n_0}\\
&\ge {\rm e}^{(\chi_E^-(\nu,g)-2\varepsilon)n_0}\\
&\ge {\rm e}^{\lambda \cdot n_0}.
\end{align*}
Similarly, we deduce that
$$
\|Dg^{n_0}|_{F(z)}\|\le \prod_{i=0}^{n_0-1}\|Dg|_{F(g^iz)}\|\le {\rm e}^{(\chi_F^+(\nu,g)+2\varepsilon)n_0}\le {\rm e}^{-\lambda \cdot n_0}.
$$
Since $\lambda>0$,  the desired result follows.
\end{proof}

\subsection{Shadowing and closing for $C^1$ diffeomorphisms without domination}
In this subsection, we prove Theorem \ref{main:shadowing}.
\begin{proof}[Proof of Theorem \ref{main:shadowing} ]

Choose $\varepsilon>0$ small enough such that
$$
0<\varepsilon<\min\left\{\frac{1}{10}\chi_E^-(\mu,f), -\frac{1}{10}\chi_F^+(\mu,f),\frac{1}{10}\left(\chi_E^-(\mu,f)-\chi_F^+(\mu,f)\right)\right\}.
$$
Let $\lambda=\min\left\{\chi_E^-(\mu,f)-\varepsilon, -(\chi_F^+(\mu,f)+\varepsilon)\right\}$.
By Proposition \ref{pro:ftoblock}, there exist  $N\in \NN$ and an ergodic component $\nu$ of   $\mu$ w.r.t. $g:=f^N$ with the following property: there exist resonance blocks $\Lambda_t^{\delta}$ ($\delta=N\varepsilon/2$) w.r.t. $(g,\nu, E\oplus F)$ such that 
$
\nu\left(\bigcup_{t\ge 1}\Lambda_t^{\delta}\right)=1.
$
Since 
$\nu,f_{\ast}\nu,\dots,f_{\ast}^{N-1}\nu$ are ergodic components of $\mu$ w.r.t. $g=f^N$, they are mutually singular.  
We may choose an $f^N$-invariant subset $R^\nu \subset \bigcup_{t \ge 1} \Lambda_t^{\delta}$ with $\nu(R^\nu) = 1$, such that the sets
\[
R^\nu, f(R^\nu), \dots, f^{N-1}(R^\nu)
\]
are pairwise disjoint. We then choose compact subsets $\widehat{\Lambda}_t \subset \Lambda_t^{\delta} \cap R^\nu$ such that
\begin{equation}\label{Lambda disjoint}
\nu(\Lambda_t^{\delta} \setminus \widehat{\Lambda}_t) < \frac{1}{t}.
\end{equation}
Hence, the following hold:
\begin{itemize}
\item[---]
$
\nu\left(\bigcup_{t\ge 1}\widehat{\Lambda}_t\right)=1.
$
\item[---] $\widehat{\Lambda}_t,\dots,f^{N-1}(\widehat{\Lambda}_t)$ are pairwise disjoint.
\end{itemize}
Let 
$$
Q_t=\bigcup_{i=0}^{N-1}f^i(\widehat{\Lambda}_t)
$$
Then 
\[\mu\left(\bigcup_{t\ge 1}Q_t\right)=1. \]
%

We show the shadowing property on $Q_t.$
Choose $\widehat{\beta}=\widehat{\beta}(t)>0$ such that 
$$
\widehat{\beta}<d(f^i(\widehat{\Lambda}_t), f^j(\widehat{\Lambda}_t)),\quad \forall 0\le i,j\le N-1, i\neq j,
$$
and
$
\|Df^{-1}\|^N\cdot \widehat{\beta} <\beta_0,
$
where $\beta_0$ is given by Theorem \ref{th:shaodowg}. Let $\widehat{N}=(N_2+1)N,$ where $N_2$ is given by Theorem \ref{th:shaodowg}.
Now fix $0 < \beta < \widehat{\beta}$, and let $\{(x_k, n_k)\}_{k \in \mathbb{Z}}$ be a $\beta$-pseudo-orbit for $f$ contained in $Q_t$ with $n_k>\widehat{N}$. For each $k \in \mathbb{Z}$, there exists a unique $p_k \in \{0, 1, \dots, N-1\}$ such that $f^{n_k}(x_k) \in f^{p_k}(\widehat{\Lambda}_t)$. Since $d(f^{n_k}(x_k), x_{k+1}) < \beta < \widehat{\beta}$ and   the minimal distance condition implies that $x_{k+1}$ must also belong to $f^{p_k}(\widehat{\Lambda}_t)$. Thus   $f^{-p_k}(x_{k+1})\in \widehat{\Lambda}_t$ and hence  $f^{-p_{k-1}}(x_k)\in \widehat{\Lambda}_t$.
    It follows from $f^{-p_{k-1}}(x_k), f^{n_k-p_k}(x_k)\in \widehat{\Lambda}_t\subset R^{\nu}$,  $R^{\nu}=f^N(R^{\nu})$  and the sets $R^{\nu}, f(R^{\nu}),\dots, f^{N-1}(R^{\nu})$ are pairwise disjoint that
$$
N\Big|n_k-p_k+p_{k-1}.
$$
Note that 
$$
d(f^{n_k-p_k}(x_k), f^{-p_k}(x_{k+1}))\le \|Df^{-1}\|^N\cdot d(f^{n_k}(x_k), x_{k+1})<\|Df^{-1}\|^N\cdot \beta<\beta_0.
$$
This shows that $\left\{(f^{-p_{k-1}}(x_k), \dfrac{n_k-p_k+p_{k-1}}{N})\right\}_{k\in \ZZ}$ is a $\|Df^{-1}\|^N\cdot \beta$ pseudo-orbit for $g$ in $\widehat{\Lambda}_t\subset \Lambda_t^{\delta}$. Let 
$$
 \widehat{n}_k=\frac{n_k-p_k+p_{k-1}}{N}>N_2, \quad \widehat{x}_k=f^{-p_{k-1}}(x_k).
$$ 
 
Observe that 
$$
0<\delta=\frac{N\varepsilon}{2}<\min\left\{\frac{1}{10}\chi_E^-(\nu,g), -\frac{1}{10}\chi_F^+(\nu,g),\frac{1}{10}\left(\chi_E^-(\nu,g)-\chi_F^+(\nu,g)\right)\right\}.
$$
Applying Theorem \ref{th:shaodowg} to $g$ and $\Lambda_t^{\delta}$, there exists a unique $\widehat{z}\in M$ such that for every $0\le j \le \widehat{n}_k$ such that 
$$
d(g^{\widehat{s}_k+j}(\widehat{z}), g^j(\widehat{x}_k))\le C_1\cdot \beta \cdot {\rm e}^{-\lambda\cdot N\cdot \min\{j,\widehat{n}_k-j\}},
$$
where 
$$
\displaystyle
\widehat{s}_k=
\left\{
\begin{array}{ll}
  \sum_{i=0}^{k-1}\widehat{n}_i, &\quad \textrm{if}~~k>0, \\[2ex]
 \quad 0, &\quad \textrm{if}~~k=0,\\[2ex]
  -\sum_{i=k}^{-1}\widehat{n}_i, & \quad \textrm{if}~~k<0. 
\end{array}
\right.
$$
Let $z=f^{p_{-1}}(\widehat{z})$. Then, for every $0\le \ell \le n_k-p_k$, there exists $0\le j \le \widehat{n}_k$ such that $$jN\le \ell +p_{k-1}<(j+1)N,$$ and     we have the estimate
\begin{align*}
d(f^{s_k+\ell}(z), f^{\ell}(x_k))&=d\left(f^{\ell+p_{k-1}}\circ g^{\widehat{s}_k}(\widehat{z}), f^{\ell+p_{k-1}}(\widehat{x}_k)\right)\\
&\le \|Df\|^N\cdot d\left(g^j\circ g^{\widehat{s}_k}(\widehat{z}), g^j (\widehat{x}_k)\right)\\
&\le C_1\cdot \|Df\|^N\cdot \beta \cdot {\rm e}^{-\lambda\cdot N\cdot \min\{j,\widehat{n}_k-j\}}\\
&\le  {C_2}\cdot \beta\cdot {\rm e}^{-\lambda\cdot \min\{\ell,n_k-\ell\}},
\end{align*}
where $C_2=C_1\cdot \|Df\|^N\cdot {\rm e}^{\lambda\cdot N}$. For $n_k-p_k< \ell \le n_k$, 
\begin{align*}
	d(f^{s_k+\ell}(z), f^{\ell}(x_k))&\le   \|Df\|^N\cdot d(f^{s_k+n_k-p_k}(z), f^{n_k-p_k}(x_k))\\
	&\le  \|Df\|^N\cdot{C_2}\cdot \beta\cdot {\rm e}^{-\lambda\cdot \min\{n_k-p_k,p_k\}}\\
	&\le \widehat{C}\cdot  \beta\cdot {\rm e}^{-\lambda\cdot \min\{\ell,n_k-\ell\}},
\end{align*}
where $\widehat{C}=C_2\cdot \|Df\|^N\cdot {\rm e}^{\lambda\cdot N}$. Therefore, for every $0\le \ell \le n_k$, 
\begin{align*}
	d(f^{s_k+\ell}(z), f^{\ell}(x_k)) \le \widehat{C}\cdot  \beta\cdot {\rm e}^{-\lambda\cdot \min\{\ell,n_k-\ell\}}.
\end{align*}
This completes the proof.
\end{proof}
 

\section{Horseshoe's approximation}
In this section, we prove Theorem \ref{main:horseshoe}.
 
\begin{proof}[Proof of theorem \ref{main:horseshoe}]
	For any sufficient small $\varepsilon>0,$ by Proposition \ref{pro:ftoblock}, there exist  $N\in \NN$ and an ergodic component $\nu$ of   $\mu$ w.r.t. $g:=f^N$ with the following property: there exist resonance blocks $\Lambda_t^{N\varepsilon/2}$   w.r.t. $(g,\nu, E\oplus F)$ such that 
	$
	\nu\left(\bigcup_{t\ge 1}\Lambda_t^{N\varepsilon/2}\right)=1.
	$
	We may choose an $f^N$-invariant subset $R^\nu \subset \bigcup_{t \ge 1} \Lambda_t^{N\varepsilon/2}$ with $\nu(R^\nu) = 1$, such that the sets
	\[
	R^\nu, f(R^\nu), \dots, f^{N-1}(R^\nu)
	\]
	are pairwise disjoint. 
And then choose compact subsets $\widehat{\Lambda}_t \subset \Lambda_t^{N\varepsilon/2}\bigcap R^\nu,$ as in \eqref{Lambda disjoint} such that $\widehat{\Lambda}_t\subset \Lambda_t^{N\varepsilon/2}$ and $\nu(\widehat{\Lambda}_t)>0$.

Denote by $d_n(x,y)=\max_{0\le k \le n-1}d(f^kx,f^ky)$ the dynamical distance on $M$. Let $B_n(x,\rho)=\{y:d_n(x,y)<\rho\}$ be the $d_n$-ball centered at $x$ of radius $\rho$. Let $N_{\mu}(n,\rho,\delta)$ be the minimal number of $d_n$-balls whose union has measure at least $\delta$. By Katok \cite{Ka80}, for every $\delta>0$,
$$
h_{\mu}(f)=\lim_{\rho\to 0}\liminf_{n\to +\infty}\frac{1}{n}\log N_{\mu}(n,\rho,\delta).
$$

 For any fixed $t\ge 1$, let 
$\Omega_t=\widehat{\Lambda}_t\cap {\rm supp}(\mu)$. Take $\delta=\frac{1}{2}\mu(\Omega_t)=\frac{1}{2}\mu(\widehat{\Lambda}_t)>0$. Then there exists $\rho_1>0$, $N_3\ge 1$ such that for every $0<\rho <\rho_1$ and $n\ge N_3$, it holds that 
\begin{equation}\label{enest1}
N_{\mu}(n,\rho,\delta)\ge {\rm e}^{(h_{\mu}(f)-\varepsilon)n}.
\end{equation}
Recall that the distance $D$ of $\cM_f(M)$ is given by
$$
D(\mu_1,\mu_2)=\sum_{j=1}^{\infty}\frac{|\int \varphi_jd\mu_1-\int \varphi_jd\mu_2|}{2^j}, \quad \forall \mu_1,\mu_2\in \cM_f(M),
$$
where $\{\varphi_j\}_{j=1}^{\infty}$ is the set of dense subset of the unit sphere in $C(M)$.
Take $J$ sufficiently large such that $1/2^J<\varepsilon/8$. Then we choose 
$$
0<\rho <\min\left\{\frac{\rho_1}{2\widehat{C}},~\widehat{\beta},~\frac{\varepsilon}{2\widehat{C}}\right\},
$$
to satisfy
\begin{equation}\label{phi J}
|\varphi_j(x)-\varphi_j(y)|\le \frac{\varepsilon}{4}, \quad \forall 1\le j \le J,
\end{equation}
whenever $d(x,y)\le \widehat{C}\cdot \rho$, where $\widehat{\beta},\widehat{C}$ are constant given in Theorem \ref{main:shadowing}.
Let $\cP=\{P_1,\dots,P_{l}\}$ be a partition of $\Omega_t$, with ${\rm diam}(P_i)\le \rho$ for each $1\le i \le l$.

Let
\begin{align*}
\Omega_{t,n} := \Bigg\{ x\in\Omega_t: 
& \exists\, n\le k\le (1+\varepsilon)n~\text{ such that } f^k(x)\in\mathcal{P}(x), \\
& \{f^k(x)\}_{1\leq k\leq n} \text{ is }\varepsilon/2\text{-dense in}~ {\rm supp}(\mu), \text{ and } \\
& \left|\frac{1}{m}\sum_{i=0}^{m-1}\varphi_j(f^ix)-\int\varphi_j\;d\mu\right| \leq\frac{\varepsilon}{4},\quad 
  \forall m\geq n,\quad 1\leq j\leq J \Bigg\}.
\end{align*}
Observe that $\lim_{n\to \infty}\mu(\Omega_{t,n})=\mu(\Omega_t)$ (see \cite[P.17]{CZ23} for instance). Take 
$$
n>\max\left\{\frac{ \log l}{\varepsilon},\frac{2\log t}{\varepsilon},\frac{1}{6\varepsilon}, \widehat{N}\right\}
$$
such that $\mu(\Omega_{t,n})>\frac{1}{2}\mu(\Omega_t)=\delta$ and $(n\varepsilon+1)<{\rm e}^{\varepsilon n}$, where $\widehat{N}$ is given by Theorem \ref{main:shadowing}.
Let $E$ be an $(n,2\widehat{C}\rho)-$ separated set of $\Omega_{t,n}$ of maximal cardinality. Then
$$
\bigcup_{x\in E}B_n(x,2\widehat{C}\rho)\supset \Omega_{t,n}.
$$
 By \eqref{enest1},  
  \[  {\#}(E)\geq N_\mu(n,2\widehat{C}\rho,\mu(\Omega_{t,n}))\geq N_\mu(n,2\widehat{C}\rho,\delta)\geq {\rm e}^{(h_\mu(f)-\varepsilon)n}. \]
  For $ k\in [n, (1+\varepsilon)n]$, let $F_k=\{x\in E:f^k(x)\in \mathcal{P}(x)\}$. Then  by the definition of  $\Omega_{t,n},$
  \[E=\bigcup_{k\in [n, (1+\varepsilon)n]}F_k.\]
   Choose $m\in [n, (1+\varepsilon)n]$ such that $\#(F_m)=\max \{\#(F_k):n\leq k\leq n(1+\varepsilon) \}.$  Then  
   \[\#(F_m)\geq \frac{1}{n\varepsilon+1}\#(E)\geq  \frac{1}{n\varepsilon+1}{\rm e}^{(h_\mu(f)-\varepsilon)n}\geq  {\rm e}^{(h_\mu(f)-2\varepsilon)n}.\]
   Take $P\in \mathcal{P}$ such that $\#(F_m\cap P)=\max \{\#(F_m\cap P_i):1\leq i\leq l\}.$ Then
   \[\#(F_m\cap P)\geq \frac{1}{l}\#(F_m)\geq  {\rm e}^{(h_\mu(f)-3\varepsilon)n}.\]
   Without loss of generality, we may  assume 
   \begin{equation}\label{40}
   	 {\rm e}^{(h_\mu(f)-3\varepsilon)n}\leq \#(F_m\cap P)\leq {\rm e}^{(h_\mu(f)+3\varepsilon)n}.
   \end{equation}
   Indeed, if this inequality does not hold, observe that  $${\rm e}^{(h_\mu(f)+3\varepsilon)n}- {\rm e}^{(h_\mu(f)-3\varepsilon)n}={\rm e}^{(h_\mu(f)-3\varepsilon)n}({\rm e}^{6\varepsilon n}-1)>1,$$
    so we may always select a subset of  $F_m\cap P$ for which \eqref{40} is satisfied.
    
    Now we construct the  desired horseshoe. Let $X=(F_m\cap P)^\mathbb{Z}.$  
Given that $\text{diam}(P) \leq \rho < \widehat{\beta}$, one can associate to any $\bar{x} = (x_i)_{i \in \mathbb{Z}} \in X$ a unique $\rho$-pseudo-orbit   $\{(x_i, m)\}_{i \in \mathbb{Z}}$ in $\widehat{\Lambda}_t$. Note that $\widehat{\Lambda}_t\subset Q_t=\bigcup_{i=0}^{N-1}f^i(\widehat{\Lambda}_t)$. By Theorem \ref{main:shadowing}, there exists a unique point $\pi(\bar{x})\in M$ whose orbit ($\widehat{C}\rho,\lambda$)-shadows   $\{(x_i, m)\}_{i \in \mathbb{Z}}$ for some $\lambda>0$, which thereby defines    a map $\pi:X\to M$. We now prove the injectivity of $\pi$. Let $\{{(x_i, m)}\}_{i \in \mathbb{Z}}$ and $\{{(\widetilde{x}_i, m)}\}_{i \in \mathbb{Z}}$ be two $\rho$-pseudo-orbits that are both $(\widehat{C}\rho,\lambda)$-shadowed by the same orbit $\mathcal{O}(y)$. Suppose, for contradiction,  that $x_{i_0}\neq \widetilde{x}_{i_0}$ for some $i_0\in \mathbb{Z}.$ Then by Theorem \ref{main:shadowing}, 
\[d_m\left(x_{i_0},f^{i_0m}(y)\right)\le \widehat{C}\rho, \quad d_m\left(\widetilde{x}_{i_0},f^{i_0m}(y)\right)\le \widehat{C}\rho. \]
It follows that $d_m(x_{i_0},\widetilde{x}_{i_0})\le 2\widehat{C}\rho$, which contradicts the fact that $x_{i_0},\widetilde{x}_{i_0}\in F_m$ are $(n,2\widehat{C}\rho)$-separated.

 Let $K^*=\pi(X).$  Then $f^m|_{K^*}$ is conjugate to   a full shift in $\#(F_m\cap P)$-symbols. Let $$K_\varepsilon=\bigcup_{i=0}^{m-1}f^i(K^*).$$
 Then $K_\varepsilon$ is a horseshoe. We show that $K_\varepsilon$ satisfies the conclusions \eqref{ho1}-\eqref{ho4} in Theorem \ref{main:horseshoe}.
 
 \begin{enumerate}[(1)]
 \item Since
 $$h_{\text {\em top}}(f|_{{K_\varepsilon}})=\frac{1}{m}h_{\text {\em top}}(f^m|_{K^*})=\frac{1}{m}\log \#(F_m\cap P),$$
 by \eqref{40} and $n\le m\le (1+\varepsilon)n$,  we obtain
 $$h_{\text {\em top}}(f|_{K_\varepsilon})\geq    \frac{n}{m}(h_\mu(f)-3\varepsilon)\ge h_\mu(f)-( h_\mu(f)+3)\varepsilon,$$
 and 
 \[ h_{\text {\em top}}(f|_{K_\varepsilon})\leq   \frac{n}{m}(h_\mu(f)+3\varepsilon)\leq h_\mu(f)+3\varepsilon.\]
 
\item  By the construction of $K_\varepsilon$,  for any $y\in K_\varepsilon,$ there exist $x\in F_m\cap P$ and $ 0\leq k\leq m-1$  such that  $d(y,f^kx)\leq \widehat{C}\rho<\varepsilon/2.$ It follows from $x\in \Omega_t\subset\text{supp}(\mu)$ that $K_\varepsilon\subset B(\text{supp}(\mu), \varepsilon/2).$  On the other hand, given any $\widetilde{x}\in F_m\cap P\subset \Omega_{t,n}$, by the construction of $\Omega_{t,n}$,  for any $z\in \text{supp}(\mu)$, there exists $1\leq k\leq n$  such that $d(z,f^k\widetilde{x})\leq \varepsilon/2$.
 Take $(x_i)_{i \in \mathbb{Z}} \in X$ with $x_0=\widetilde{x}$.  Let $y\in K^*$ be a point whose orbit $\widehat{C}\rho$-shadows   $\{(x_i, m)\}_{i \in \mathbb{Z}}$. Then $d(f^k\widetilde{x},f^ky)\leq \widehat{C}\rho<\varepsilon/2$. Therefore, $d(z,f^ky)< \varepsilon$, which implies $\text{supp}(\mu)\subset B(K_\varepsilon, \varepsilon/2).$   Hence $d_H(K_\varepsilon,\text{supp}(\mu))< \varepsilon.$

\item For any $f$-invariant measure $\nu$ supported on $K_\varepsilon$, we may first assume that $\nu$ is ergodic. Since 
 $$D(\mu,\nu)\leq \sum\limits_{j=1}^{J}\frac{|\int \varphi_jd\mu-\int \varphi_jd\nu|}{2^j}+\frac{1}{2^{J-1}}\leq \sum\limits_{j=1}^{J}\frac{|\int \varphi_jd\mu-\int \varphi_jd\nu|}{2^j}+\frac{\varepsilon}{4},$$
it suffices to show that $|\int \varphi_jd\mu-\int \varphi_jd\nu|\leq \dfrac{3}{4}\varepsilon$ for all $ 1\leq j\leq J.$ Take $y \in K^*$ and $s\ge 1$ sufficiently large such that
\[
\left| \frac{1}{ms} \sum_{k=0}^{ms-1} \varphi_j(f^k y) - \int \varphi_j \, d\nu \right| \leq \frac{\varepsilon}{4}, \quad \forall 1 \leq j \leq J.
\]
Then  by the construction of $K^*$, there exist points $(x_i)_{i \in \mathbb{Z}} \in X= (F_m \cap P)^\mathbb{Z}$ such that
\[
d(f^{mi + j} y, f^j x_i) \leq \widehat{C}\rho, \quad \forall \ 0 \leq j \leq m,~ i\in\mathbb{Z}.
\]

By \eqref{phi J} and the construction of $\Omega_{t,n}$, we conclude that
\[
\left| \int \varphi_j \, d\mu - \int \varphi_j \, d\nu \right| \leq \frac{3}{4} \varepsilon, \quad  \forall 1 \leq j \leq J,
\]
which implies $D(\mu, \nu) \leq \varepsilon$.
 
In general, if $\nu$ is not ergodic, then by the ergodic decomposition theorem, $\nu$-almost every ergodic component is supported on $K_\varepsilon$. Consequently,
 \begin{align*}
 	D(\mu,\nu)  &= \sum\limits_{j=1}^{\infty}\frac{1}{2^j}\left|\int \varphi_jd\mu-\int \varphi_jd\nu\right|\\ 
 	& =  \sum\limits_{j=1}^{\infty}\frac{1}{2^j}\left|\int\left(\int \varphi_jd\mu-\int \varphi_jd\nu_x\right)d\nu(x)\right| \\
 	&\leq   \int D(\mu,\nu_x) d\nu(x)  \\
 	& \leq \varepsilon.
 \end{align*}

\item For any $x\in F_m \cap P\subset \Omega_{t,n},$   it holds that $f^m(x)\in P\subset \widehat{\Lambda}_t.$     By the definition of  $\widehat{\Lambda}_t,$ one has  $\widehat{\Lambda}_t \subset \Lambda_t^{N\varepsilon/2}\bigcap R^\nu.$    Since $R^{\nu}$ is an $f^N$-invariant subset  such that $R^{\nu}, f(R^{\nu}), \dots, f^{N-1}(R^{\nu})$ are pairwise disjoint, we conclude that $N|m.$ Writing $m = kN$ and applying Corollary~\ref{cor:lambda-bl} to $g = f^N$ and $\Lambda_t^{N\varepsilon/2}$, we conclude that
\begin{align*}
	m(Df^m|_{E(x)}) & \ge \prod_{i=0}^{k-1}m\left(Df^N|_{E(f^{iN}x)}\right)\\
	                                &\ge t^{-1}\cdot {\rm e}^{(\chi_E^-(\mu,f)- \varepsilon/2)kN}\\
	                                &\ge {\rm e}^{(\chi_E^-(\mu,f)-\varepsilon)m},
\end{align*}
and 
\begin{align*}
	\|Df^m|_{F(x)}\| & \le \prod_{i=0}^{k-1}\|Df^N|_{F(f^{iN}x)}\|\\
	&\le t \cdot {\rm e}^{(\chi_F^+(\mu,f)+\varepsilon/2)kN}\\
	&\le {\rm e}^{(\chi_F^+(\mu,f)+\varepsilon)m}.
\end{align*}
\end{enumerate}
This completes the proof of  Theorem \ref{main:horseshoe}.
\end{proof}

\section{Approximation of Lyapunov exponents}
Let $f\in {\rm Diff}^1(M)$ and $\mu$ be a hyperbolic ergodic measure.  Suppose that  the Oseledets splitting $T_{\Gamma}M=E_1\oplus \cdots \oplus E_{\ell}$ is extendable continuous. Denote by   $\chi_1 > \cdots > \chi_\ell $   the  Lyapunov exponents of $
\mu$.  Let 
$$
\varepsilon_0=\min_{1\le i\le \ell}\left\{\frac{|\chi_i|}{10},\frac{| \chi_{i+1}-\chi_i|}{10} \right\}
$$ 
and we may also assume $B(\supp(\mu),\varepsilon_0/2)\subset \mathcal{U}$.   For any $0<\varepsilon<\varepsilon_0$ and $N\ge 1$, let   
\begin{align*}
	a^1_n( \varepsilon,x,  N)=&\sup_{\substack{0\le k \le n\\1\le j\le \ell}}\left\{  \frac{\prod_{i=n-k}^{n-1}\|Df^N|_{E_j(f^{iN}x)}\|}{{\rm e}^{(\chi_j+\varepsilon)kN}}, ~~ \frac{{\rm e}^{(\chi_j -\varepsilon)kN}}{   \prod_{i=n-k}^{n-1}m(Df^N|_{E_j(f^{iN}x)})}\right\}, \\
	a^2_n(\varepsilon,x,  N)=&\sup_{\substack{k\ge 0\\1\le j\le \ell}}\left\{  \frac{\prod_{i=n}^{n+k-1}\|Df^N|_{E_j(f^{iN}x)}\|}{{\rm e}^{(\chi_j+\varepsilon)kN} }, ~~ \frac{{\rm e}^{(\chi_j -\varepsilon)kN}}{   \prod_{i=n }^{n+k-1}m(Df^N|_{E_j(f^{iN}x)})}\right\}.
\end{align*}
where we set $\prod_{i=n}^{n-1}a_i=1$ for any sequence $\{a_i\}$ of real numbers.
For any $t\ge 1$, define
$$
H_t(\varepsilon,x,  N)=\left\{n\in \NN: a^i_n(\varepsilon,x,  N)\le t, \quad \forall 1\le i \le 2 \right\}.
$$
and 
$$
\Lambda_t(\varepsilon,x,  N)=\bigcap_{m\ge 1}\overline{\bigcup_{j\ge m}\left\{g^j(x):j\in H_t(\varepsilon,x,  N)\right\}}.
$$
Then similar to the proof of   Proposition  \ref{pro:ftoblock},   there exist $N\in \NN$, an ergodic component $\nu$ of $\mu$ w.r.t. $f^N$ and  $x_0\in M $  such that     
$$
\nu\left(\bigcup_{t\ge 1}\Lambda_t(\varepsilon,x_0,  N)\right)=1.
$$

For simplicity, denote $g=f^N$, $\Lambda_t=\Lambda_t(\varepsilon,x_0,  N)$ and let $\mathcal{R}^\nu=\bigcup_{t\ge 1}\Lambda_t$. For $x\in \mathcal{R}^\nu$, define 
$$
T(x)=\inf\left\{t\ge 1: x\in \Lambda_t\right\}.
$$
Then the same argument of  Lemma \ref{T measurable} and  \ref{lem:temp} gives 
\begin{lemma}
	$T$ is measurable on $\mathcal{R}^\nu$ and   $T$ is tempered, that is, 
	$$
	\lim_{n\to \pm\infty}\frac{1}{n}\log T(g^{n}(x))=0,\quad \nu-\textrm{a.e.}~x\in \mathcal{R}^\nu.
	$$
\end{lemma}
Moreover, similar to the proof of Lemma \ref{pro:fulles}, $T(x)$ also satisfies the following property.
\begin{proposition}\label{Liaopesin}
	For every $x\in \mathcal{R}^\nu$, $k\ge 1$ and $1\le j\le \ell$, one has
	\begin{enumerate}[(1)]
		\item\label{la:p1} 
		$
	 T(x)^{-1}\cdot {\rm e}^{(\chi_j-\varepsilon)kN}\le 	\prod_{i=-k}^{-1}m(Dg|_{E_j(g^i(x))})\le     \prod_{i=-k}^{-1}\|Dg|_{E_j(g^i(x))}\|\le T(x)\cdot {\rm e}^{(\chi_j+\varepsilon)kN},
		$
		\item\label{la:p2}
		$
	T(x)^{-1}\cdot {\rm e}^{(\chi_j-\varepsilon)kN}\le	\prod_{i=0}^{k-1}m(Dg|_{E_j(g^i(x))})  \le \prod_{i=0}^{k-1}\|Dg|_{E_j(g^i(x))}\|\le T(x)\cdot {\rm e}^{(\chi_j+\varepsilon)kN}.
		$
	\end{enumerate}
	
\end{proposition}
 We are now ready to prove Theorem \ref{periodic approximation}.
 \begin{proof}[Proof of Theorem \ref{periodic approximation} ]
 		Let  $\varepsilon, N,g,\nu$ be as above.   Choose $r_0>0$ small enough  such that for any points $z_1, z_2\in \mathcal{U}$ with $d(z_1,z_2)\le r_0$, it holds that 
 	 \begin{equation}\label{42} 
 	 	\frac{\|Dg|_{E_j(z_1)}\|}{\|Dg|_{E_j(z_2)}\|},~ \frac{m(Dg|_{E_j(z_1)})}{m(Dg|_{E_j(z_2)})}\in \left({\rm e}^{-\varepsilon N}, {\rm e}^{\varepsilon N}\right),\quad \forall 1\le j\le \ell.
 	 \end{equation}
  Given any $t\ge 1$, let  $\beta_0$, $C_1$, $N_2$ be constants given by  Theorem \ref{th:shaodowg}. Define 
  \[\Omega_t=\left\{x\in \Lambda_t: \text{there exists} ~n_k\to \infty \text{~such that ~} f^{n_k}(x)\in \Lambda_t ~\text{and}~ f^{n_k}(x)\to x\right \}. \]
  Then by Poincar\'e recurrence theorem, $\nu(\Omega_t)=\nu(\Lambda_t)>0$.  Take $x\in \Omega_t.$ Then we can choose $n>\max\left\{N_2, \dfrac{\log t}{\varepsilon N}   \right\}$ such that  
  \[f^n(x)\in \Lambda_t~\text{and~}\beta:=d(f^nx,x)<\min\left\{\beta_0, {r_0/c_1}\right\},\]
 where  $\beta_0$, $C_1$, $N_2$ are constants given by  Theorem \ref{th:shaodowg}.   Then by   Theorem \ref{th:shaodowg}, there exists $p=f^n(p)$ such that
$$
d(g^{i}(p), g^i(x))\le C_1\cdot \beta, \quad \forall 0\le i\le n.
$$ 
Therefore, for any $1\le j\le \ell$, by Proposition \ref{Liaopesin} and equation  \eqref{42}, 

\begin{align*}
	\|Dg^n|_{E_j(p)}\| & \le \prod_{i=0}^{n-1}\|Dg|_{E_j(g^{i}p)}\|\\
	&\le \prod_{i=0}^{n-1}\left({\rm e}^{\varepsilon N}\cdot\|Dg|_{E_j(g^{i}x)}\|\right)\\
	&\le t\cdot {\rm e}^{(\chi_j+2\varepsilon)nN }\\
	&\le   {\rm e}^{(\chi_j+3\varepsilon)nN },
\end{align*}
   and 
  \begin{align*}
  	m(Dg^n|_{E_j(p)}) & \ge \prod_{i=0}^{n-1}m\left(Dg|_{E_j(g^{i}p)}\right)\\
  	&\ge \prod_{i=0}^{n-1}\left({\rm e}^{-\varepsilon N}\cdot m(Dg|_{E_j(g^{i}x)})\right)\\
  	&\ge t^{-1}\cdot {\rm e}^{(\chi_j-2\varepsilon)nN}\\
  	&\ge {\rm e}^{(\chi_j-3\varepsilon)nN}.
  \end{align*}
Denote by  $\lambda_1(\mu_p)\ge \cdots \ge \lambda_d(\mu_p)$ the Lyapunov exponents of $
 \mu_p$ counted with multiplicities. Then for $\sum_{i=1}^{j-1}\dim(E_i)<k\le \sum_{i=1}^{j}\dim(E_i)$, we conclude that 
 \[\chi_j-3\varepsilon\le \lambda_k(\mu_p)\le \chi_j+3\varepsilon. \]
 This completes the proof.
 \end{proof}
 
 At last, we provide a  sketch of the  proof of Theorem \ref{horseshoe approximation}.
 \begin{proof}[Sketch of proof of Theorem \ref{horseshoe approximation}]
 	 The proof closely follows that of Theorem \ref{main:horseshoe}. The key modification is to replace the estimates for  $\|Df^m|_{E(x)}\|$ and $m(Df^m|_{F(x)})$ in item (4) with the corresponding estimates for $\|Df^m|_{E_j(x)}\|$ and $m(Df^m|_{F_j(x)})$, which are provided by using Proposition \ref{Liaopesin}. As the remainder of the argument is standard, we omit the details.
 \end{proof}

\vskip 5pt

\flushleft{\bf Yongluo Cao} \\
\small Department of Mathematics and Statistics, and 
Center for Dynamical Systems and Differential Equation, Soochow University, Suzhou 215006, Jiangsu, P.R. China\\
\textit{E-mail:} \texttt{ylcao@suda.edu.cn}\\

\flushleft{\bf Zeya Mi} \\
\small School of Mathematics and Statistics, 
Nanjing University of Information Science and Technology, Nanjing, 210044, P.R. China\\
\textit{E-mail:} \texttt{mizeya@163.com}\\

\flushleft{\bf Rui Zou} \\
\small School of Mathematics and Statistics, 
Nanjing University of Information Science and Technology, Nanjing, 210044, P.R. China\\
\textit{E-mail:} \texttt{zourui@nuist.edu.cn}\\
\end{sloppypar}
\end{document}